\documentclass[fleqn]{zzz}
\usepackage{mathrsfs}
\usepackage{color}
\usepackage{xcolor}
\usepackage{hyperref}
\usepackage{fancyvrb} 
\usepackage{xparse}
\usepackage{float}

\def\real{{\mathbb R}}
\def\intg{{\mathbb Z}}

\hypersetup{
   colorlinks = true,
   urlcolor = blue,
   linkcolor = red
}

\numberwithin{equation}{section} 

\newcommand{\OmeH}[2]{
\Omega_{#1}^{d}
\left(
l_{#1},#2
;\!\begin{array}{c} {\theta_{#1}} \\[3pt] {\theta_{#1}'}
\end{array}
\right)}
\newcommand{\Omeh}[2]{
\Omega_{#1}^{d}
\left(
l,#2
;\!\begin{array}{c} {\theta} \\[0pt] {\theta'}
\end{array}
\right)}

\newcommand{\upsH}[3]{
\Upsilon_{#1}^{#2}
\left(
\begin{array}{c} {n_{#1}} \\ {#3}
\end{array};
{\vartheta_{#1}} 
\right)}
\newcommand{\upSH}[3]{
\Upsilon_{#1}^{#2}
\left(
\begin{array}{c} {n} \\ {#3}
\end{array};
{\vartheta} 
\right)}
\newcommand{\upsHp}[3]{
\Upsilon_{#1}^{#2}
\left(
\begin{array}{c} {n_{#1}} \\ {#3}
\end{array};
{\vartheta_{#1}'} 
\right)}

\newcommand{\PsiH}[3]{
\Psi_{#1}^{#2}
\left(
\begin{array}{c} {n_{#1}} \\ {#3}
\end{array};
\begin{array}{c} {\vartheta_{#1}} \\[5pt] {\vartheta_{#1}'}
\end{array}
\right)}
\newcommand{\PsIH}[3]{
\Psi_{#1}^{#2}
\left(
\begin{array}{c} {n} \\ {#3}
\end{array};
\begin{array}{c} {\vartheta} \\ {\vartheta'}
\end{array}
\right)}

\newcommand{\rrRRp}{\left(\frac{rr'}{RR'}\right)^p}
\newcommand{\pth}{\left(\frac{r_>^2-r_<^2}{2rr'}\right)^{p+\frac32}}
\newcommand{\pdh}{\left(\frac{r_>^2-r_<^2}{2rr'}\right)^{p+\frac{d-1}{2}}}

\newcommand{\hyp}[5]{{}_{#1}F_{#2}\!\left(\genfrac{}{}{0pt}{}{#3}{#4};#5\right)}

\newcommand{\bfx}{{\bf x}}
\newcommand{\bfxp}{{{\bf x}^\prime}}
\newcommand{\wbfx}{\widehat{\bf x}}
\newcommand{\wbfxp}{{\widehat{\bf x}^\prime}}
\newcommand{\N}{{\mathbb N}}

\newcommand{\Q}{{\mathbb Q}}
\newcommand{\R}{{\mathbb R}}
\newcommand{\Si}{{\mathbf S}}

\newcommand{\li}{{\mathfrak l}}

\newcommand{\Z}{{\mathbb Z}}
\newcommand{\C}{{\mathbb C}}
\newcommand{\expe}{{\mathrm e}}
\newcommand{\const}{{\mathrm{const}}}

\newcommand{\fdt}{{d/2}}
\newcommand{\fdf}{{d/4}}
\newcommand{\gdt}{{\frac{d}{2}}}

\newcommand{\mcM}{{\mathcal M}}
\newcommand{\mcN}{{\mathcal N}}
\newcommand{\mcL}{{\mathcal L}}

\newcommand{\mcg}{{\mathcal G}}

\newcommand{\hii}{{\mathfrak h}}
\newcommand{\jii}{{\mathfrak j}}

\newtheorem{thm}[lemma]{Theorem}
\newtheorem{cor}[lemma]{Corollary}
\newtheorem{lem}[lemma]{Lemma}
\newtheorem{rem}[lemma]{Remark}

\numberwithin{equation}{section}
\numberwithin{corollary}{section}
\numberwithin{remark}{section}
\numberwithin{theorem}{section}
\numberwithin{lemma}{section}

\begin{document}

\renewcommand{\PaperNumber}{***}

\FirstPageHeading

\ShortArticleName{Binomial and logarithmic polyharmonic expansions and addition theorems}

\ArticleName{Gegenbauer expansions and addition theorems for a binomial and logarithmic fundamental solution of the 
even-dimensional Euclidean polyharmonic equation}

\Author{Howard S.~Cohl$^\ast$, Jessie E.~Hirtenstein$^\dag$,
Jim Lawrence\,$^{\ddag\S}$
and Lisa Ritter\,$^\S$}

\AuthorNameForHeading{H.~S.~Cohl, J.~E.~Hirtenstein, J.~Lawrence, L.~Ritter}


\Address{$^\dag$~Applied and Computational Mathematics Division,
National Institute of Standards and Technology,
Mission Viejo, CA 92694, USA
\URLaddressD{\href{http://www.nist.gov/itl/math/msg/howard-s-cohl.cfm}{http://www.nist.gov/itl/math/msg/howard-s-cohl.cfm}}
} 
\EmailD{howard.cohl@nist.gov} 

\Address{$^\dag$~Department of Physics, University of California, Davis, CA 95616, USA}
\EmailD{jhirtenstein@ucdavis.edu} 

\Address{$^\ddag$~Department of Mathematical Sciences, 
George Mason University, Fairfax, VA 22030, USA}
\EmailD{lawrence@gmu.edu}

\Address{$^\S$~Applied and Computational Mathematics Division,
National Institute of Standards and Technology,
Gaithersburg, MD 20878, USA
} 
\EmailD{lisa.ritter@nist.gov} 

\ArticleDates{Received ???, in final form ????; Published online ????}

\Abstract{On even-dimensional Euclidean space for integer powers of the Laplace operator greater than
or equal to half the dimension, a fundamental solution of the polyharmonic
equation has binomial and logarithmic behavior. Gegenbauer polynomial expansions of these fundamental
solutions are obtained through a limit  applied to Gegenbauer expansions of a power-law 
fundamental solution of the polyharmonic equation. This limit is accomplished through parameter differentiation. By combining these results with previously 
derived azimuthal Fourier series expansions for these binomial and logarithmic fundamental 
solutions, we are able to obtain addition theorems for the azimuthal Fourier coefficients. 
These logarithmic and binomial addition theorems are expressed in Vilenkin polyspherical geodesic polar coordinate systems and as well in generalized Hopf coordinates on 
spheres in arbitrary even dimensions.
}


\Classification{35A08, 31B30, 31C12, 33C05, 42A16}

\section{Introduction}
\label{Introduction}

Analysis of the polyharmonic operators (natural powers of the Laplace operator)
are ubiquitous in many areas of pure
and applied mathematics and as well in
physics and engineering problems.
Here we concern ourselves with a fundamental solution of the 
polyharmonic equation (Laplace, biharmonic, etc.), which is connected
to solutions of the 
inhomogeneous polyharmonic equation.  
Solutions to inhomogeneous polyharmonic equations are useful in many 
physical applications including those areas related to Poisson's equation
such as Newtonian gravity, electrostatics, magnetostatics, quantum
direct and exchange interactions
\cite[\S 1]{CohlDominici}, etc.
Furthermore, applications of higher-powers of the Laplace operator
include such varied areas as minimal surfaces \cite{MonterdeUgail}, continuum 
mechanics \cite{LaiRubinKrempl},
mesh deformation \cite{Helenbrook}, elasticity \cite{LurieVasiliev},
Stokes flow \cite{Kirby}, geometric design \cite{Ugail},
cubature formulae \cite{Sobolev}, mean value theorems 
(cf.~Pizzetti's formula) \cite{Nicolescu}, and Hartree-Fock calculations 
of nuclei \cite{Vautherin}.

Closed-form expressions for the Fourier expansions of a logarithmic fundamental 
solution for the polyharmonic equation are extremely useful when solving
inhomogeneous polyharmonic problems on even-dimensional Euclidean space, especially
when a degree of rotational symmetry is involved.  
A fundamental solution of the polyharmonic equation on $d$-dimensional 
Euclidean space $\R^d$ has two arguments and therefore maps from a $2d$-dimensional 
space to the reals.  
Solutions to the inhomogeneous polyharmonic equation can be 
obtained by convolution of a fundamental solution with an integrable source distribution.
Eigenfunction decompositions of a fundamental solution reduces the dimension of the 
resulting convolution integral to obtain Dirichlet boundary values in order to solve 
the resulting elliptic system, replacing it instead by a sum or an integral over some 
parameter space.  By taking advantage of rotational or nearly rotational symmetry in 
the azimuthal Fourier decomposition of the source distribution, one reduces the 
dimensionality of the resulting convolution integral and obtains a rapidly convergent 
Fourier cosine expansion.  In the case of an axisymmetric (constant angular dependence) 
source distribution, the entire contribution to the boundary values are determined by a 
single term in the azimuthal Fourier series.  These kinds of expansions have been 
previously shown to be extremely effective in solving inhomogeneous problems (see for 
instance the discussion \cite{CT} and those papers which cite it).  

It is well-known (see for 
\cite[p.~45]{Schw}, \cite[p.~202]{GelfandShilov})
that a fundamental solution of the polyharmonic equation
on $d$-dimensional Euclidean space $\R^d$ is given by combinations of 
power-law and logarithmic functions of the global distance between two points.
In \cite{CohlDominici},
the Fourier coefficients of a power-law fundamental solution
of the polyharmonic equation were obtained. Gegenbauer and Jacobi polynomial 
expansions of a power-law fundamental solution 
of the polyharmonic equation (which generalize the Fourier expansions presented in \cite{CohlDominici}) were obtained
in \cite{Cohl12pow}.
The coefficients of these expansions
are seen to be given in terms of associated Legendre functions.
Fourier expansions of a logarithmic and binomial
fundamental solution of the polyharmonic
equation were obtained in
\cite{CohllogFourier}.

The work presented in this manuscript
is concerned with computing the Gegenbauer  
coefficients of binomial and logarithmic fundamental solutions of the 
polyharmonic equation.  One obtains a logarithmic fundamental 
solution for the polyharmonic equation only on even-dimensional
Euclidean space and only when the power of the Laplace operator 
is greater than or equal to the dimension divided by two. 
The most familiar example of a logarithmic fundamental solution of
the polyharmonic equation occurs in two-dimensions, for a 
single-power of the Laplace operator, i.e., Laplace's equation.
In this manuscript we present an approach for obtaining the Gegenbauer expansion
of binomial and logarithmic fundamental solutions of the polyharmonic equation by 
parameter differentiation of a power-law fundamental solution of the polyharmonic
equation.

This manuscript is organized as follows.  
In Section 
\ref{Specialfunctionsandorthogonalpolynomials} we
introduce the fundamental mathematical sets, sequences, functions and orthogonal polynomials,
which are necessary to understand the mathematical details of this manuscript.
In Section \ref{Prelude} we describe the properties
of binomial and logarithmic fundamental solutions of the polyharmonic
equation in even-dimensional Euclidean space 
$\R^d$---and in particular in rotationally-invariant coordinate systems and in Vilenkin polyspherical coordinates.
In Section \ref{GegTnexpansions}, we derive
using a limit-derivative approach, the
Fourier cosine and Gegenbauer polynomial expansions
for the kernels which occur in binomial
and logarithmic fundamental solutions of the 
polyharmonic equation in Euclidean space. These
kernels are $(z-x)^p$ and $(z-x)^p\log(z-x)$,
where $z>1$, $x\in(-1,1)$ and $p\in\N_0$.
In Section \ref{FundExpansions} we use the results
presented in the previous sections to obtain the azimuthal 
Fourier and Gegenbauer expansions for binomial
and logarithmic fundamental solution of the 
polyharmonic equation
in rotationally-invariant coordinate systems
in even-dimensional Euclidean space.
In Section \ref{AdditionthmsinVilenkinspolysphericalcoordinates}
we derive addition theorems for the azimuthal
Fourier coefficients for the binomial and
logarithmic fundamental solutions
of the polyharmonic equation in
of even-dimensional Euclidean space.

\section{{Preliminaries}}
\label{Specialfunctionsandorthogonalpolynomials}

{
Here we introduce some nomenclature which
we will rely upon in the text below.
Throughout this paper we adopt the following set notations:
$\mathbb N_0:=\{0\}\cup\mathbb N=\{0, 1, 2, 3, \ldots\}$, and
we use the sets $\Z$, $\R$, $\mathbb C$ to represent the integers, real and complex
numbers respectively.
As is the common convention for associated Legendre 
functions \cite[(8.1.1)]{Abra}, for any expression of the 
form $(z^2-1)^\alpha$, read this as 
$(z^2-1)^\alpha:=(z+1)^\alpha(z-1)^\alpha$,
for any fixed $\alpha\in\mathbb C$ and $z\in\mathbb C\setminus(-\infty,1]$.
Over the set of complex numbers, we assume that empty sum vanishes and the empty product is unity.
Given two numbers: $r,r'\in\R$, define
\begin{equation}
r_{\lessgtr} := \substack{\text{min} \\ \text{max}} \{ r, r' \}.
\end{equation}
}

The harmonic number $H_n\in\Q$, $n\in\Z$, given by 
\begin{equation}
H_n:=\left\{\begin{array}{ll} 0,\quad &\mbox{if}\  n\le 0,\\[0.2cm]
{\displaystyle \sum_{k=1}^n\frac1k}, \quad&\mbox{if}\ n \ge 1.
\end{array}\right.
\label{harmnum}
\end{equation}
The harmonic number is related to the polygamma function \cite[(5.2.2)]{NIST:DLMF} with $n\in\N_0$ since
\cite[p.~14]{MOS}
\[
\psi(n+1)=-\gamma+H_n,
\]
where $\gamma\approx 0.57721\ldots$ is the Euler-Mascheroni constant \cite[(5.2.3)]{NIST:DLMF}.\medskip

The Pochhammer symbol is defined for $a\in\mathbb C$, $n\in{\mathbb N}$,
such that
\begin{equation}
\label{2:1}
(a)_0:=1,\quad (a)_n:=(a)(a+1)\cdots(a+n-1).
\end{equation}
One also has
\begin{equation}
(-p)_n=\left\{
\begin{array}{cl}
{\displaystyle \frac{(-1)^pp!}{(p-n)!}} & \mbox{if}\ 0\le n\le p,\\[0.4cm]
0 & \mbox{if}\ n\ge p+1,
\end{array}
\right.
\label{negintpoch}
\end{equation}
Note for $z\in\C$, $k\in\N_0$, the binomial coefficient can be given in terms of the Pochhammer
symbol as follows \cite[(1.2.6)]{NIST:DLMF}
\begin{equation}
\label{binomz}
\binom{z}{k}=\frac{(-1)^k (-z)_k}{k!}.
\end{equation}

The Gauss hypergeometric function 
${}_2F_1:\C\times\C\times(\C\setminus-\N_0)\times
\C\setminus[1,\infty)$ 
\cite[Chapter 15]{NIST:DLMF} is defined as
\[
\hyp21{a,b}{c}{z}
=\sum_{n=0}^\infty \frac{(a)_n(b)_n}{(c)_n}\frac{z^n}{n!},
\]
for $|z|<1$ and through analytical continuation for the rest of its domain.
If one takes $b=c$ in the Gauss hypergeometric function, 
one produces the binomial theorem \cite[(15.4.6)]{NIST:DLMF}
\begin{equation}
\hyp10{a}{-}{z}=(1-z)^{-a}.
\label{binomthm}
\end{equation}
Associated Legendre functions of the first and second kind
$P_\nu^\mu,Q_\nu^\mu:\C\setminus(-\infty,1]\to\C$
\cite[(14.3.6-7) and Section 14.21]{NIST:DLMF} are defined as
\begin{equation}
P_\nu^\mu(z):=\frac{1}{\Gamma(1-\mu)}\left(\frac{z+1}{z-1}\right)^{\frac{\mu}{2}}
\hyp21
{-\nu,\nu+1}{1-\mu}
{\frac{1-z}{2}},
\label{associatedLegendrefunctionP}
\end{equation}
for $|1-z|<2$,
\begin{equation}
Q_\nu^\mu(z):=\frac{\sqrt{\pi}\expe^{i\pi\mu}\Gamma(\nu+\mu+1)(z^2-1)^{\frac{\mu}{2}}}
{2^{\nu+1}\Gamma(\nu+\frac32)z^{\nu+\mu+1}}
\hyp21
{
\frac{\nu+\mu+1}{2},
\frac{\nu+\mu+2}{2}}
{\nu+\frac32}
{\frac{1}{z^2}},
\label{defnQLegendre}
\end{equation}
for $|z|>1$, and elsewhere in $z$ by analytic continuation
of the 
Gauss hypergeometric function.
A property of the associated Legendre functions that we will take advantage of 
are the Whipple formulae \cite[(14.9.16-17)]{NIST:DLMF} (these are equivalent to each other)
\begin{eqnarray}
&&
Q_\nu^\mu(z)=
\sqrt{\frac{\pi}{2}}\expe^{i\mu\pi}\Gamma(\nu+\mu+1)(z^2-1)^{-\frac14}
P_{-\mu-\frac12}^{-\nu-\frac12}\left(\frac{z}{\sqrt{z^2-1}}\right),\label{Whip1}\\[0.0cm]
&& P_\nu^\mu(z)=
i\sqrt{\frac{2}{\pi}}
\frac{\expe^{i\nu\pi}(z^2-1)^{-\frac14}}
{\Gamma(-\nu-\mu)}
Q_{-\mu-\frac12}^{-\nu-\frac12}\left(\frac{z}{\sqrt{z^2-1}}\right),
\label{Whip2}
\label{whipple}
\end{eqnarray}
for $\Re z>0$.
These allow one to convert between the 
associated Legendre functions of the first and second kind. We also take advantage of 
\cite[(14.9.14)]{NIST:DLMF}
\begin{equation}
Q_\nu^{-\mu}(z)=\expe^{-2i\pi\mu}\frac{\Gamma(\nu-\mu+1)}{\Gamma(\nu+\mu+1)}Q_\nu^\mu(z).
\label{LegendreQmorder}
\end{equation}
Some useful special cases are (cf.~\cite[(14.5.17)]{NIST:DLMF})
\begin{equation}
Q_{\frac12}^\frac12(z)=\frac{i\sqrt{\pi/2}}{(z^2-1)^{\frac14}(z+\sqrt{z^2-1})},
\label{Qhh}
\end{equation}
and
\begin{equation}
Q_{n+\frac12}^\frac32(z)=\frac{
-i\sqrt{\pi/2}(z+(n+1)\sqrt{z^2-1})}
{(z^2-1)^{\frac34}(z+\sqrt{z^2-1})^{n+1}},
\label{Qnhnp1}
\end{equation}
which follows from \cite[(14.10.6)]{NIST:DLMF} and \cite[(14.5.17)]{NIST:DLMF}.
One also has
\begin{equation}
Q_{\mu-\frac12}^{\mu+\frac12}(z)=\frac{\expe^{i\pi(\mu+\frac12)}2^{\mu-\frac12}\Gamma(\mu+\frac12)}
{(z^2-1)^{\frac{\mu}{2}+\frac14}},
\label{Qmhmph}
\end{equation}
where in this formula $\mu\in\C\setminus\{-\frac12,-1,-\frac32,\ldots\}$,
which follows using \eqref{defnQLegendre} and the binomial theorem \eqref{binomthm}.
Let $p,m\in\N_0$.

In regard to parameter differentiation of associated Legendre functions of the first kind, then 
Szmytkowski derived \cite[cf.~(5.12)]{Szmy1} 
\begin{eqnarray}
&&\hspace{-1.6cm}\left[\frac{\partial}{\partial\nu}P_\nu^m(z)\right]_{\nu=p}\nonumber\\[0.2cm]
&&\hspace{-0.5cm}=\left\{
\begin{array}{ll}
\hspace{-0.0cm}
{\displaystyle \left(
\log\frac{z+1}{2}+2H_{2p}-H_p-H_{p-m}\right)P_p^m(z)}&\\[0.4cm]
\hspace{0.5cm}{}+{\displaystyle (-1)^{p+m}\sum_{k=0}^{p-m-1}(-1)^k\frac{(2k+2m+1)
\left[1+\frac{k!(p+m)!}{(k+2m)!(p-m)!}\right]
}{(p-m-k)(p+m+k+1)}P_{k+m}^m(z) }&\\[0.4cm]
\hspace{0.5cm}{}+{\displaystyle (-1)^{p}\frac{(p+m)!}{(p-m)!}\sum_{k=0}^{m-1}\frac{(-1)^k(2k+1)}{(p-k)(p+k+1)}
P_k^{-m}(z)},& \mbox{if}\ 0\le m\le p,\\[0.6cm]
\hspace{0.0cm}(-1)^{p+m+1}(p+m)!(m-p-1)!P_p^{-m}(z), & \mbox{if}\ m\ge p+1.
\end{array}
\right.
\label{limitderivdegreepnum}
\end{eqnarray}
One interesting consequence of Szmytkowski's formula is the following corollary which does not seem to have appeared elsewhere in the literature.
\begin{cor}
Let $p,m\in\N_0$, $0\le m\le p$, $z\in\C\setminus(-\infty,1]$.  Then
\begin{eqnarray}
&&\left[\frac{\partial}{\partial\nu}Q_{m-\frac12}^{\nu+\frac12}(z)\right]_{\nu=p}
=\left[2H_{2p}-H_p-\gamma+i\pi+\log\frac{z+\sqrt{z^2-1}}{2\sqrt{z^2-1}}\right]
Q_{m-\frac12}^{p+\frac12}(z)\nonumber\\[0.2cm]
&&\hspace{1cm}+(p-m)!\sum_{k=0}^{p-m-1}\frac{(2k+2m+1)\left[1+\frac{k!(p+m)!}{(k+2m)!(p-m)!}\right]}
{k!(p-m-k)(p+m+k+1)}Q_{m-\frac12}^{k+m+\frac12}(z)\nonumber\\[0.2cm]
&&\hspace{1cm}+(p+m)!\sum_{k=0}^{m-1}\frac{2k+1}{(k+m)!(p-k)(p+k+1)}Q_{m-\frac12}^{k+\frac12}(z).
\end{eqnarray}
\end{cor}
\begin{proof}
Starting with \eqref{Whip2} for $\mu=m\in\N_0$, and the reflection formula for the gamma function
\cite[(5.5.3)]{NIST:DLMF}, we have
\[
P_\nu^m(z)
=\frac{-i}{(z^2-1)^\frac14}\sqrt{\frac{2}{\pi}}
\frac{\expe^{-i\pi\nu}}{\Gamma(\nu-m+1)}
Q_{m-\frac12}^{\nu+\frac12}\left(\frac{z}{\sqrt{z^2-1}}\right).
\]
Then replace
\begin{eqnarray}
&&\hspace{-1.2cm}\left[\frac{\partial}{\partial\nu}P_\nu^m(z)\right]_{\nu=p}=
\frac{-i}{(z^2-1)^\frac14}\sqrt{\frac{2}{\pi}}
\left[
\frac{\partial}{\partial\nu}
\frac{\expe^{-i\pi\nu}}{\Gamma(\nu-m+1)}
Q_{m-\frac12}^{\nu+\frac12}\left(\frac{z}{\sqrt{z^2-1}}\right)\right]_{\nu=p}\nonumber\\[0.2cm]
&&\hspace{-0.9cm}=\frac{-i}{(z^2-1)^\frac14}\sqrt{\frac{2}{\pi}}
\left\{
\left[\frac{\partial}{\partial\nu}
\frac{\expe^{-i\pi\nu}}{\Gamma(\nu-m+1)}\right]_{\nu=p}
Q_{m-\frac12}^{p+\frac12}\left(\frac{z}{\sqrt{z^2-1}}\right)+
\frac{(-1)^p}{(p-m)!}
\left[\frac{\partial}{\partial\nu}
Q_{m-\frac12}^{\nu+\frac12}\left(\frac{z}{\sqrt{z^2-1}}\right)\right]_{\nu=p}
\right\},\nonumber
\end{eqnarray}
with
\[
\left[\frac{\partial}{\partial\nu}
\frac{\expe^{-i\pi\nu}}{\Gamma(\nu-m+1)}\right]_{\nu=p}=\frac{(-1)^{p+1}}{(p-m)!}\left(i\pi+\psi(p-m+1)\right),
\]
in \eqref{limitderivdegreepnum} for $0\le m\le p$, together with \eqref{Whip2}, \cite[(14.9.14)]{NIST:DLMF}, 
and applying the map $z\mapsto \frac{z}{\sqrt{z^2-1}}$
\cite[Appendix A]{Cohlonpar},
completes the proof.
Unfortunately, applying this same method for $m\ge p+1$ does not seem to produce an 
analogous result.
\end{proof}

\noindent On the other hand, using the map $z\mapsto\frac{z}{\sqrt{z^2-1}}$, \eqref{limitderivdegreepnum},
and \eqref{Whip1}, produces the following useful relation
\begin{eqnarray}
&&\hspace{-0.2cm}\left[\frac{\partial}{\partial\nu}P_\nu^m\left(\frac{z}{\sqrt{z^2-1}}\right)\right]_{\nu=p}
=i\sqrt{\frac{2}{\pi}}(z^2-1)^{\frac14}
\nonumber\\[0.2cm]
&&\hspace{0.4cm}\times\left\{
\begin{array}{ll}
\hspace{-0.0cm}
{\displaystyle \frac{(-1)^{p+1}}{(p-m)!}
\left(
\log\frac{z+\sqrt{z^2-1}}{2\sqrt{z^2-1}}+2H_{2p}-H_p-H_{p-m}\right)Q_{m-\frac12}^{p+\frac12}(z)}&\\[0.4cm]
\hspace{0.5cm}{}+{\displaystyle 
(-1)^{p+1}\sum_{k=0}^{p-m-1}\frac{(2k+2m+1)
}{k!(p-m-k)(p+m+k+1)}
\left[1+\frac{k!(p+m)!}{(k+2m)!(p-m)!}\right]
Q_{m-\frac12}^{k+m+\frac12}(z) }&\\[0.4cm]
\hspace{0.5cm}{}+{\displaystyle 
\frac{(-1)^{p+1}(p+m)!}{(p-m)!}
\sum_{k=0}^{m-1}\frac{2k+1}{(k+m)!(p-k)(p+k+1)}
Q_{m-\frac12}^{k+\frac12}(z)},& \mbox{if}\ 0\le m\le p,\\[0.6cm]
\hspace{0.0cm}(-1)^{m}(m-p-1)!\ Q_{m-\frac12}^{p+\frac12}(z), & \mbox{if}\ m\ge p+1.
\end{array}
\right.
\label{limitderivdegreeqnum}
\end{eqnarray}

One main orthogonal polynomial that we use in this manuscript is the Gegenbauer polynomial.
The Gegenbauer polynomial $C_n^\nu:\C\to\C$ can be defined by
\begin{equation}
C_n^\nu(x):=\frac{(2\nu)_n}{n!}
\hyp21{-n,n+2\nu}{\nu+\frac12}{\frac{1-x}{2}},
\label{gegpolydefn}
\end{equation}
where $n\in\N_0$, $\nu\in(-\frac12,\infty)\setminus\{0\}.$
{The Gegenbauer polynomial can also be written in terms of the Ferrers function of the first kind (associated Legendre function of the first kind on-the-cut) \cite[(18.11.1)]{NIST:DLMF}
\begin{align} \label{JacobiGeg}
    C_{n-m}^{m+\frac12}(x) =
    \frac{1}{(\frac12)_m(-2)^m(1-x^2)^{\frac{m}{2}}} {\sf P}_{n}^m (x),
\end{align}
where $x\in(-1,1)$, $n,m\in\N_0$ such that 
$0\le m\le n$.
}
{
The Ferrers function of the first kind is
defined as \cite[(14.3.1)]{NIST:DLMF}
\begin{equation}
{\sf P}_\nu^\mu(x)=\frac{1}{\Gamma(1-\mu)}
\left(\frac{1+x}{1-x}\right)^{\frac12\mu}
\hyp21{-\nu,\nu+1}{1-\mu}{\frac{1-x}{2}},
\end{equation}
where $\nu,\mu\in\C$, $x\in\C\setminus((-\infty,-1]\cup[1,\infty))$.
}
One may consider the limit as $\mu\to 0$ of the Gegenbauer polynomial
\cite[(6.4.13)]{AAR}
\begin{equation}
\lim_{\mu\to 0}\frac{{n}+\mu}{\mu}C_{n}^\mu(x)=\epsilon_{n} T_{n}(x)
\label{limitGegChebyI}
\end{equation}
where $\epsilon_{n}=2-\delta_{{n},0}$ is the Neumann factor, commonly
appearing in Fourier cosine series.
The Chebyshev polynomial of the first kind $T_{n}:\C\to\C,$ is defined by
\[
T_n(x):=
\hyp21{-n,n}{\frac12}{\frac{1-x}{2}},
\]
with the useful Fourier cosine representation
$
T_n(\cos\theta)=\cos(n\theta).
$
The Chebyshev polynomial of the second kind $U_n:\C\to\C$, is defined by 
\cite[(18.7.4)]{NIST:DLMF} $U_n(x):=C_n^{1}(x).$

For generalized Hopf coordinates (see Section \ref{Ghopfsec} below) we use Jacobi polynomials which are defined as \cite[(18.5.7)]{NIST:DLMF}
\[
P_n^{(\alpha,\beta)}(z)=\frac{(\alpha+1)_n}{n!}
\hyp21{-n,\alpha+\beta+n+1}{\alpha+1}{\frac{1-z}{2}}.
\]
\begin{thm}
Let $\alpha,\beta,n\in\N_0$, $z\in\C$. Then the Jacobi polynomials have the following symmetry 
in parameter relations
\begin{eqnarray}
&& P_n^{(\alpha,\beta)}(z)=\frac{(\alpha+n)!(\beta+n)!}{n!(\alpha+\beta+n)!}\left(\frac{z-1}{2}\right)^{-\alpha}P_{\alpha+n}^{(-\alpha,\beta)}(z),\label{Jacobinega} \\[0.2cm]
&& P_n^{(\alpha,\beta)}(z)=\frac{(\alpha+n)!(\beta+n)!}{(\alpha+\beta+n)!n!}\left(\frac{z+1}{2}\right)^{-\beta}P_{\beta+n}^{(\alpha,-\beta)}(z), \label{Jacobinegb}
\end{eqnarray}
\end{thm}
\begin{proof}
For \eqref{Jacobinega}, start with \cite[p.~212]{MOS}
\[
P_n^{(\alpha,\beta)}(z)=\binom{n+\beta}{n}\left(\frac{z-1}{2}\right)^n
\hyp21{-n,-n-\alpha}{\beta+1}{\frac{z+1}{z-1}},
\]
replace $\alpha\mapsto-\alpha$, then $n\mapsto n+\alpha$. Then using \eqref{binomz}, and the reflection 
formula \cite[(5.5.3)]{NIST:DLMF}, the final result follows since for $x,y\in\Z$, one has 
\[
\lim_{\epsilon\to0}\frac{\sin(\pi(x+y+2\epsilon))}{\sin(\pi(x+\epsilon))}=\cos(\pi y)=(-1)^y.
\]
For \eqref{Jacobinegb} use \cite[(18.5.8)]{NIST:DLMF}
\[
P_n^{(\alpha,\beta)}(z)=\frac{(\alpha+1)_n}{n!}\left(\frac{z+1}{2}\right)^n
\hyp21{-n,-n-\beta}{\alpha+1}{\frac{z-1}{z+1}},
\]
and replace first $\beta\mapsto-\beta$, and then $n\mapsto n+\beta$.
\end{proof}

\begin{cor}
Let $\alpha,\beta\in\Z$, $n\in\N_0$, $z\in\C$. Then the Jacobi polynomials have the following symmetry 
in parameter relation
\begin{eqnarray}
&& P_n^{(\alpha,\beta)}(z)= \left(\frac{z-1}{2}\right)^{-\alpha}\left(\frac{z+1}{2}\right)^{-\beta}P_{\alpha+\beta+n}^{(-\alpha,-\beta)}(z).\label{Jacobinegab}
\end{eqnarray}
\end{cor}
\begin{proof}
For \eqref{Jacobinegab} use \eqref{Jacobinega} and then substitute \eqref{Jacobinegb} to replace the 
Jacobi polynomial which appears on the right-hand side. Notice that the restriction 
$\alpha,\beta\in\N_0$ is relaxed to $\alpha,\beta\in\Z$ in the resulting expression.
\end{proof}

\section{Binomial and logarithmic kernels for the even dimensional fundamental solution of the polyharmonic equation in Euclidean space}
\label{Prelude}

If $\bfx,\bfxp\in\R^d$ then the Euclidean inner product
$(\cdot,\cdot):\R^d\times\R^d\to\R$ defined by
\begin{equation}
(\bfx,\bfxp):=x_1x_1'+\cdots+x_dx_d',
\end{equation}
induces a norm (the Euclidean norm) $\|\cdot\|:\R^d\to[0,\infty)$, on $\R^d$,
given by $\|\bfx\|:=\sqrt{(\bfx,\bfxp)}$. 
If $\Phi:\R^d\to\R$ satisfies the polyharmonic equation given by
\begin{equation}
(-\Delta)^k\Phi(\bfx)=0,
\label{polyharmoniceq}
\end{equation}
where 
$k\in\N$ and $\Phi\in C^{2k}(\R^d),$
$\bfx\in\R^d,$ and $\Delta:C^p(\R^d)\to C^{p-2}(\R^d)$ for $p\ge 2$, is the
Laplace operator on $\mathbb R^d$ defined by
$\Delta:=\frac{\partial^2}{\partial x_1^2}+\cdots+\frac{\partial^2}{\partial x_d^2}.$
Then $\Phi$ is referred to as 
polyharmonic, and $(-\Delta)^k$ is referred to as the polyharmonic operator. 
If the power $k$ of the Laplace operator equals two, then \eqref{polyharmoniceq} is referred to as the 
biharmonic equation and $\Phi$ is called biharmonic.
The inhomogeneous polyharmonic equation is given by
\begin{equation}
(-\Delta)^k\Phi({\bf x})=\rho({\bf x}),
\label{polyh}
\end{equation}
where we take $\rho$ to be an integrable function so that 
a solution to \eqref{polyh} exists.  A fundamental solution for the 
polyharmonic equation on $\R^d$ is a function 
${\mcg}_k^d:(\R^d\times\R^d)\setminus\{(\bfx,\bfx):\bfx\in\R^d\}\to\R$ 
which satisfies the distributional equation
\begin{equation}
(-\Delta)^k{\mcg}_k^d({\bf x},{\bf x}^\prime)=\delta({\bf x}-{\bf x}^\prime),
\label{unnormalizedfundsolnpolydefn}
\end{equation}
where $\delta$ is the Dirac delta distribution
and $\bfxp\in\R^d$.\medskip


Let $n\in\N$ and define the harmonic number $H_n:=\sum_{k=1}^n\frac{1}{k}$.
Then a fundamental solution of the polyharmonic equation \eqref{polyharmoniceq} 
on Euclidean space $\R^d$ is given by
(see for instance 
\cite[(2.1)]{Boyl},
\cite[Theorem 1]{CohllogFourier},
\cite[Section II.2]{Sobolev})
\begin{equation}
\mcg_k^d({\bf x},{\bf x}^\prime):=
\left\{ \begin{array}{ll}
{\displaystyle \frac{(-1)^{k+\frac{d}{2}+1}\,\|{\bf x}-{\bf x}^\prime\|^{2k-d}}
{(k-1)!\,\left(k-\frac{d}{2}\right)!\,2^{2k-1}\pi^{\frac{d}{2}}}
\left(\log\|{\bf x}-{\bf x}^\prime\|-\beta_{k-\frac{d}{2},d}\right),}
\hspace{0.659cm} \mathrm{if}\  d\,\,\mathrm{even},\ k\ge \frac{d}{2},\\[10pt]
{\displaystyle \frac{\Gamma(\frac{d}{2}-k)\|{\bf x}-{\bf x}^\prime\|^{2k-d}}
{(k-1)!\,2^{2k}\pi^{\frac{d}{2}}}} \hspace{5.82cm} \mathrm{otherwise},
\end{array} \right. 
\label{greenpoly}
\end{equation}
where $\beta_{p,d}\in\Q$ is defined as 
$\beta_{p,d}:=\frac12(H_p+H_{\frac{d}{2}+p-1}-H_{\frac{d}{2}-1}).$

\begin{rem}
\noindent In regard to a logarithmic fundamental solution of the polyharmonic
equation ($d$ even, $k\ge \frac{d}{2}$), note that \cite[Section II.2]{Sobolev} is missing the term proportional
to $\|\bfx-\bfxp\|^{2k-d}$. This term is in the kernel of the
polyharmonic operator $(-\Delta)^k$, so for any constant multiple
of this term $\beta_{p,d}$ may be added to a fundamental solution of the polyharmonic
equation.  Our choice for this constant is given so that
\begin{equation}
-\Delta \mcg_k^d=\mcg_{k-1}^d,
\label{iter}
\end{equation}
is satisfied for all $k> \frac{d}{2}$, and that for $k=\frac{d}{2}$, the constant vanishes.
Boyling's fundamental solution satisfies \eqref{iter} for all $k>\frac{d}{2}$,
but is missing the term proportional to $H_{\frac{d}{2}-1}$,
and therefore only vanishes when $k=\frac{d}{2}$ for $d=2$.  Sobolev does
not include this constant term, so for him $\mcg_k^d$ is purely logarithmic
for all $k\ge \frac{d}{2}$, $d\ge 2$ even. However in that case \eqref{iter} is not 
satisfied for $k>\frac{d}{2}$.
\end{rem}


\medskip

We consider parametrizations of Euclidean space $\R^d$ which are given 
in terms of coordinate systems whose coordinates are curvilinear, i.e.,
based on some transformation which converts the Cartesian coordinates to a coordinate system 
with the same number of coordinates in which the coordinate lines are curved.  We consider 
solutions of the polyharmonic equation \eqref{polyharmoniceq} in a curvilinear coordinate
system, which arises through the theory of separation of variables.  We refer to coordinate 
systems which yield solutions 
through the separation of variables method as separable.  In this manuscript, we restrict our 
attention to separable rotationally-invariant coordinate systems for the polyharmonic 
equation on $\R^d$ which are given by
{
\begin{equation}
\left.
\begin{array}{rcl}
x_1&=&x_1(\xi_1,\ldots,\xi_{d-1})\\[0.1cm]
&\vdots&\\[0.1cm]
x_{d-2}&=&x_{d-2}(\xi_1,\ldots,\xi_{d-1})\\[0.1cm]
x_{d-1}&=&R(\xi_1,\ldots,\xi_{d-1})\cos\phi\\[0.1cm]
x_d&=&R(\xi_1,\ldots,\xi_{d-1})\sin\phi\\[0.1cm]
\end{array}
\quad\right\}.
\label{rotatioanallyinvariant}
\end{equation}
}
These coordinate systems are described by $d$ coordinates: an angle 
$\phi\in\R$ plus $(d-1)$-curvilinear coordinates $(\xi_1,\ldots,\xi_{d-1})$.
Rotationally-invariant coordinate systems parametrize points on the $(d-1)$-dimensional
half-hyperplane given by $\phi=\const$ and $R\ge 0$ using the curvilinear coordinates
$(\xi_1,\ldots,\xi_{d-1})$.
A separable rotationally-invariant coordinate system transforms the polyharmonic equation 
into a set of $d$-uncoupled ordinary differential equations with separation constants 
$m\in\Z$ and $k_j$ for $1\le j\le d-2$.  
For a separable rotationally-invariant 
coordinate system, this uncoupling
is accomplished, in general, by assuming a solution to \eqref{polyharmoniceq} of the form
\[
\Phi(\bfx)=\expe^{im\phi}\,{\mathcal R}(\xi_1,\ldots,\xi_{d-1})\prod_{i=1}^{d-1} 
A_i(\xi_i,m,k_1,\ldots,k_{d-2}),
\]
where the properties of the functions 
${\mathcal R}$ and $A_i$, for $1\le i\le d-1$, and the constants $k_j$ for 
$1\le j\le d-2$, depend on the specific separable rotationally-invariant 
coordinate system in question.
Separable coordinate systems are divided into two separate classes, 
those which are simply separable (${\mathcal R}=\const$), and those which are ${\mathcal R}$-separable.
For an extensive description of the theory of separation of variables
see \cite{Miller}.

The Euclidean distance between two points $\bfx,\bfxp\in\R^d$, expressed in a 
rotationally-invariant coordinate system, is given by
\begin{equation}
\displaystyle \|\bfx-\bfxp\|=\sqrt{2RR^\prime}
\left[\chi-\cos(\phi-\phi^\prime)\right]^{\frac12},
\label{rotationallyinvariantdist}
\end{equation}
where the hypertoroidal parameter $\chi>1$, is given by 
\begin{equation}
\chi:=\chi(R,R',x_1,\ldots,x_{d-2},x_1',\ldots,x_{d-2}'):=\frac{R^{2}+{R^\prime}^2
+{\displaystyle \sum_{k=1}^{d-2}(x_k-x_k^\prime)^2}
}
{\displaystyle 2RR^\prime},
\label{chidefn}
\end{equation}
where $R,R^\prime\in(0,\infty)$ are defined in \eqref{rotatioanallyinvariant}.
The hypersurfaces given by 
$\chi=\const$
are independent of
coordinate system and represent hypertori of revolution. 

One type of coordinate system which parametrizes points in $d$-dimensional Euclidean
space which has a high degree of symmetry are Vilenkin's polyspherical 
coordinates (for a detailed description of Vilenkin's polyspherical coordinates, see \cite[Appendix B and references therein]{Cohl12pow}).
These curvilinear orthogonal coordinate systems are composed of a radius 
$r\in[0,\infty)$ and $(d-1)$ angles which must have domains given in 
$\{[0,\frac12\pi],[0,\pi],[-\pi,\pi)\}$.
Using these coordinate systems, we can also express the distance between 
two points as 
\begin{equation}
\|\bfx-\bfxp\|=\sqrt{2rr'}(\zeta-\cos\gamma)^{\frac12},
\label{sphericallysymmetricdist}
\end{equation}
where $\zeta:[0,\infty)^2\to [1,\infty)$ is defined by 
\begin{equation}
\zeta=\zeta(r,r'):=\frac{r^2+r'^2}{2rr'},
\label{zetadef}
\end{equation}
and the separation angle $\gamma\in[0,\pi]$ is defined through the relation
\begin{equation}
\cos\gamma:=\frac{(\bfx,\bfxp)}{\|\bfx\|\|\bfxp\|}=\frac{(\bfx,\bfxp)}{rr'},
\label{sepang}
\end{equation}
using the Euclidean inner product and norm.
Note that
\begin{equation} \label{zetasquare}
\sqrt{\zeta^2 - 1} = \frac{r_>^2-r_<^2}{2rr'}.
\end{equation}

From \eqref{greenpoly} we see that apart from a multiplicative constant,
the algebraic expression 
of a fundamental solution for the polyharmonic equation 
on Euclidean space $\R^d$ for $d$ even, $k\ge \frac{d}{2},$ is given by
$\li_k^d:(\R^d\times\R^d)\setminus\{(\bfx,\bfx):\bfx\in\R^d\}\to\R$ defined by
\begin{equation}
\li_k^d(\bfx,\bfxp):=\|\bfx-\bfxp\|^{2k-d}\left(\log\|\bfx-\bfxp\|-\beta_{k-\frac{d}{2},d}\right).
\label{liunderscorekd}
\end{equation}
By expressing $\li_k^d$ in a rotationally-invariant   coordinate system
\eqref{rotatioanallyinvariant}
one has the following result.
Let $p=k-d/2\in\N_0$. Then
\begin{eqnarray}
&&\hspace{-0.7cm}\li_k^d(\bfx,\bfxp)=\left(2RR^\prime\right)^p\left(\frac12\log
\left(2RR^\prime\right)-\beta_{p,d}\right)
\left(\chi-\cos(\phi-\phi^\prime) \right)^p\nonumber\\[2pt]
&&\hspace{5cm}+\frac12\left(2RR^\prime\right)^p
\left(\chi-\cos(\phi-\phi^\prime) \right)^p
\log\left(\chi-\cos(\phi-\phi^\prime) \right).
\label{lft1}
\end{eqnarray}
Also, in a Vilenkin polyspherical coordinate system,
one has
\begin{eqnarray}
&&\hspace{-0.7cm}\li_k^d(\bfx,\bfxp)=
\left(2rr^\prime\right)^p\left(\frac12\log
\left(2rr^\prime\right)-\beta_{p,d}\right)
\left(\zeta-\cos\gamma \right)^p+\frac12\left(2rr^\prime\right)^p
\left(\zeta-\cos\gamma\right)^p
\log\left(\zeta-\cos\gamma \right).
\label{lft2}
\end{eqnarray}

For the polyharmonic equation on even-dimensional Euclidean space $\R^d$, if $1\le k \le \frac{d}{2}-1$
then a fundamental solution is given by  
{$\hii_k^d:(\R^d\times\R^d)\setminus\{(\bfx,\bfx):\bfx\in\R^d\}\to\R$ which is a power-law given by
$\hii_k^d(\bfx,\bfxp)=\|\bfx-\bfxp\|^{2k-d}$,
where $2k-d\in-2\N$.
For this range of $k$ values then $2k-d$ is a negative
even integer and this case and all its implications are fully covered by the material presented in 
\cite{Cohl12pow}. }

{
For the case in which a logarithmic
fundamental solution exists, namely 
$k \ge \frac{d}{2}$, then $2k-d\in2\N$ is a positive
even integer and the kernel 
$\jii_k^d:\R^d\times\R^d\to\R$
\begin{equation}
\jii_k^d(\bfx,\bfxp):=\|\bfx-\bfxp\|^{2k-d}.
\label{hkdii}
\end{equation}
corresponds
to a binomial expression $(z-x)^p$, where $p=k-d/2\in\N_0$ and its series expansion either 
in terms of Chebyshev polynomials of the first kind
or in terms of Gegenbauer polynomials truncates
in a finite number of terms.
By expressing $\jii_k^d$ in a rotationally-invariant coordinate system there is
\begin{equation}
\jii_k^d(\bfx,\bfxp)=\left(2RR^\prime\right)^{p}
\left[\chi-\cos(\phi-\phi^\prime) \right]^{p},
\label{unlogfourseries}
\end{equation}
and expressing $\jii_k^d$ in a Vilenkin polyspherical
coordinate system one has 
\begin{equation}
\jii_k^d(\bfx,\bfxp)=\left(2rr^\prime\right)^{p}
\left[\zeta-\cos\gamma \right]^{p},
\label{unlogVilen}
\end{equation}
where $p=k-d/2\in\N_0$.
}

\section{Fourier and Gegenbauer expansions of binomial and logarithmic kernels}
\label{GegTnexpansions}

We require Gegenbauer expansions for kernels which naturally arise in a logarithmic 
fundamental solution of the polyharmonic equation in Vilenkin's polyspherical coordinates 
in even-dimensional Euclidean space.
Since these coordinates are rotationally invariant, in this study of addition theorems
in these coordinates, we require Fourier cosine (Chebyshev polynomials 
of the first kind) and Gegenbauer 
polynomial expansions (see \cite{Cohl12pow}). By necessity, here we treat the 
binomial $(z-x)^{p}$ and logarithmic $(z-x)^p\log(z-x)$ kernels, where $x,z,\nu\in\C$ 
and $p\in\N_0$. We have previously derived Fourier expansions of the 
binomial \cite[(3.10)]{CohlDominici} and the logarithmic \cite[(20), (26)]{CohllogFourier} kernels. 
We now treat those corresponding Gegenbauer polynomial expansions.
By examining \eqref{lft1}, \eqref{lft2}, \eqref{unlogfourseries},
 we see that for 
the computation of Fourier and Gegenbauer expansions, we are interested
in the Fourier and Gegenbauer expansions of the Euler and logarithmic 
kernels.

{
The formulas which are presented below
all rely in one way or another on the 
following important Gegenbauer polynomial
expansion which can be 
found in \cite[(3.4)]{Cohl12pow}, namely
\begin{equation}
(z-x)^{-\nu}
=\frac{2^{\mu+\frac12}\Gamma(\mu)}{\sqrt{\pi}\,\Gamma(\nu)}\expe^{i\pi(\mu-\nu+\frac12)}
(z^2-1)^{\frac{\mu-\nu}{2}+\frac14}
\sum_{n=0}^\infty (n+\mu)C_n^\mu(x)
Q_{n+\mu-\frac12}^{\nu-\mu-\frac12}(z),
\label{pGf}
\end{equation}
where $z\in\C\setminus(-\infty,1]$, and 
$x\in\C$ 
lies inside the ellipse with foci at $\pm 1$ that passes through $z$.
This result and the
following results have the curious property of 
$z,x$ lying on ellipses with foci at $\pm 1$
and this is due to the important theorem of
Szeg\H{o}, namely \cite[Theorem 12.7.3, Expansion of an analytic function in terms of orthogonal polynomials]{Szego}.
}

{
\begin{lem} Let $p\in\N_0$,
$z\in\C\setminus(-\infty,1]$, $x\in\C$.
Then
the expansion of Euler kernel in Chebyshev polynomials of the
first kind 
is given by the following binomial expansion
(see \cite[(4.4)]{CohlDominici})
\begin{eqnarray}
&&(z-x)^p=i(-1)^{p+1}\sqrt{\frac{2}{\pi}}p!(z^2-1)^{\frac{p}{2}+\frac14}
\sum_{n=0}^p\frac{(-1)^n\epsilon_nT_n(x)}{(p-n)!(p+n)!}Q_{n-\frac12}^{p+\frac12}(z),
\label{Eulerpolyexp}
\end{eqnarray}
where $p\in\N_0$.
\end{lem}
}
{
\begin{proof}
This truncated series result follows from
\eqref{pGf} with $-\nu=p\in\N_0$ and then
taking $\mu\to0$ with \eqref{limitGegChebyI}.
\end{proof}
}


\noindent The Fourier cosine expansion 
of the {important} logarithmic kernel is given in the following lemma.
\begin{lem}
Let $p\in\N_0$, {$z\in\C\setminus(-\infty,1]$, 
$z-x\in\C\setminus(-\infty,0]$, $x$ lies inside the ellipse with foci at $\pm 1$ that passes through $z$.}
Then
\begin{eqnarray}
&&\hspace{-0.5cm}(z-x)^p\log(z-x)
=(z-x)^p
\left(\log\frac{z+\sqrt{z^2-1}}{2}+2H_{2p}\right)
\nonumber\\[0.2cm]
&&\hspace{+0.20cm}+i\sqrt{\frac{2}{\pi}}(-1)^pp!\left(z^2-1\right)^{\frac{p}{2}+\frac14}\sum_{n=0}^p
\frac{\epsilon_n(-1)^nT_n(x)}{(p-n)!(p+n)!}
\left(
H_{p+n}+H_{p-n}
\right)
Q_{n-\frac12}^{p+\frac12}(z)
\nonumber\\[0.2cm]
&&\hspace{+0.20cm}
+i\sqrt{\frac{2}{\pi}}(-1)^{p+1}p!\left(z^2-1\right)^{\frac{p}{2}+\frac14}\sum_{n=0}^{p-1}
\frac{\epsilon_n(-1)^nT_n(x)}{(p+n)!}
\sum_{k=0}^{p-n-1}\frac
{(2n+2k+1)\left[1+\frac{k!(p+n)!}{(2n+k)!(p-n)!}\right]}
{k!(p-n-k)(p+n+k+1)}
Q_{n-\frac12}^{n+k+\frac12}(z)
\nonumber\\[0.2cm]
&&\hspace{+0.20cm}
+ i\frac{2^{\frac32}}{\sqrt{\pi}}
(-1)^{p+1}p!\left(z^2-1\right)^{\frac{p}{2}+\frac14}\sum_{n=1}^{p}\frac{(-1)^nT_n(x)}{(p-n)!}
\sum_{k=0}^{n-1}\frac
{2k+1}
{(n+k)!(p-k)(p+k+1)}
Q_{n-\frac12}^{k+\frac12}(z)
\nonumber\\[0.2cm]
&&\hspace{+0.20cm}
+i\frac{2^{\frac32}}{\sqrt{\pi}}
p!\left(z^2-1\right)^{\frac{p}{2}+\frac14}\sum_{n=p+1}^\infty \frac{(n-p-1)!\,T_n(x)}{(p+n)!}
Q_{n-\frac12}^{p+\frac12}(z).
\label{limitderivCheby}
\end{eqnarray}
\end{lem}
\begin{proof}
Use \cite[(26)]{CohllogFourier}, with
\eqref{Whip1}, \eqref{Whip2}, \eqref{LegendreQmorder}.
\end{proof}

The following consequence of \eqref{limitderivCheby} is also given in 
\cite[p.~259]{MOS}.

\begin{cor} 
Let {$z\in\C\setminus(-\infty,1]$, $z-x\in\C\setminus(-\infty,0]$, $x$ lies inside the ellipse with foci at $\pm 1$ that passes through $z$.}
Then
\[
\log(z-x)=\log\frac{z+\sqrt{z^2-1}}{2}-2\sum_{n=1}^\infty \frac{T_n(x)}{n}
\frac{1}{(z+\sqrt{z^2-1})^n}.
\]
\end{cor}
\begin{proof}
Let $p=0$ in 
\eqref{limitderivCheby}
and use  \cite[(14.5.17)]{NIST:DLMF}.
This completes the proof.
\end{proof}

The Gegenbauer polynomial expansion of the binomial is given in the following lemma.
\begin{lem}
Let $p\in\N_0$, $\mu\in(-\frac12,\infty)\setminus\{0\}$, 
{$z\in\C\setminus(-\infty,1]$, 
$x\in\C$.}
Then
\begin{equation}
(z-x)^p=\frac{2^{\mu+\frac12}}{\sqrt{\pi}}\expe^{i\pi(p-\mu-\frac12)}\Gamma(\mu)p!
(z^2-1)^{\frac{p+\mu}{2}+\frac14}
\sum_{n=0}^p \frac{(-1)^n(n+\mu)C_n^\mu(x)}
{(p-n)!\Gamma(n+p+2\mu+1)}
Q_{n+\mu-\frac12}^{p+\mu+\frac12}(z).
\label{binomGeg}
\end{equation}
\end{lem}
\begin{proof}
Starting with \cite[(3.4)]{Cohl12pow},
\eqref{pGf},
we take the limit $\nu\to-p,$ producing
\begin{equation}
(z-x)^{p}
=\frac{2^{\mu+\frac12}}{\sqrt{\pi}}
\expe^{i\pi(\mu+p+\frac12)}
\Gamma(\mu)
(z^2-1)^{\frac{\mu+p}{2}+\frac14}
\sum_{n=0}^\infty (n+\mu)C_n^\mu(x)
\lim_{\nu\to-p}
\frac{1}{\Gamma(\nu)}
Q_{n+\mu-\frac12}^{\nu-\mu-\frac12}(z).
\label{binomqlim}
\end{equation}
By using \eqref{LegendreQmorder}, \eqref{negintpoch}, we have 
\[
\lim_{\nu\to-p}
\frac{1}{\Gamma(\nu)}
Q_{n+\mu-\frac12}^{\nu-\mu-\frac12}(z)=
\left\{
\begin{array}{cl}
{\displaystyle \frac{(-1)^{n+1}\expe^{-2i\pi\mu}p!}{(p-n)!\Gamma(n+p+2\mu+1)}Q_{n+\mu-\frac12}^{p+\mu+\frac12}(z)}
& \mbox{if}\ 0\le n\le p,\\[0.4cm]
0 & \mbox{if}\ n\ge p+1.
\end{array}
\right.
\]
Using this limit in \eqref{binomqlim} completes the proof.
\end{proof}

\noindent Note that the above lemma generalizes \eqref{Eulerpolyexp} by the {limit formula} \eqref{limitGegChebyI}.\medskip

The Gegenbauer polynomial expansion of the logarithmic kernel is given as follows.

\begin{lem} \label{biloggeg}
Let $p\in\N_0$, $\mu\in\N$, {$z\in\C\setminus(-\infty,1]$, $z-x\in\C\setminus(-\infty,0]$, $x$ lies inside the ellipse with foci at $\pm 1$ that passes through $z$.}
Then
\begin{eqnarray}
&&\hspace{-1.0cm}(z-x)^p\log(z-x)
=(z-x)^p
\left(\log\frac{z+\sqrt{z^2-1}}{2}+2H_{2p+2\mu}+H_p-H_{p+\mu}\right)
\nonumber\\[0.2cm]
&&\hspace{-0.5cm}\hspace{+0.20cm}+i(-1)^{p+\mu}\frac{2^{\mu+\frac12}}{\sqrt{\pi}}p!(\mu-1)!
\left(z^2-1\right)^{\frac{p+\mu}{2}+\frac14}\sum_{n=0}^p
\frac{(n+\mu)(-1)^nC_n^\mu(x)}{(p-n)!(p+n+2\mu)!}
\left(
H_{p+n+2\mu}+H_{p-n}
\right)
Q_{n+\mu-\frac12}^{p+\mu+\frac12}(z)
\nonumber\\[0.2cm]
&&\hspace{-0.5cm}\hspace{+0.20cm}
+i(-1)^{p+\mu+1}\frac{2^{\mu+\frac12}}{\sqrt{\pi}}p!(\mu-1)!\left(z^2-1\right)^{\frac{p+\mu}{2}+\frac14}
\sum_{n=0}^{p-1}
\frac{(n+\mu)(-1)^nC_n^\mu(x)}{(p+n+2\mu)!}\nonumber\\[-0.0cm]
&&\hspace{-0.5cm}\hspace{1.2cm}\times\sum_{k=0}^{p-n-1}\frac
{(2n+2k+2\mu+1)
}
{k!(p-n-k)(p+n+k+2\mu+1)}
\left[1+\frac{k!(p+n+2\mu)!}{(k+2n+2\mu)!(p-n)!}\right]
Q_{n+\mu-\frac12}^{k+n+\mu+\frac12}(z)
\nonumber\\[-0.0cm]
&&\hspace{-0.5cm}\hspace{+0.20cm}
+ i(-1)^{p+\mu+1}\frac{2^{\mu+\frac12}}{\sqrt{\pi}}
p!(\mu-1)!\left(z^2-1\right)^{\frac{p+\mu}{2}+\frac14}\sum_{n=0}^{p}\frac{(n+\mu)(-1)^nC_n^\mu(x)}{(p-n)!}
\nonumber\\[-0.2cm]
&&\hspace{-0.5cm}\hspace{1.2cm}\times\sum_{k=0}^{n+\mu-1}\frac
{2k+1}
{(n+k+\mu)!(p+\mu-k)(p+k+\mu+1)}
Q_{n+\mu-\frac12}^{k+\frac12}(z)
\nonumber\\[0.2cm]
&&\hspace{-0.5cm}\hspace{+0.20cm}
+i(-1)^\mu\frac{2^{\mu+\frac12}}{\sqrt{\pi}}
p!(\mu-1)!\left(z^2-1\right)^{\frac{p+\mu}{2}+\frac14}\sum_{n=p+1}^\infty \frac{(n+\mu)(n-p-1)!C_n^\mu(x)}
{(p+n+2\mu)!}
Q_{n+\mu-\frac12}^{p+\mu+\frac12}(z).
\label{limitderivGeg}
\end{eqnarray}
\end{lem}
\begin{proof}
Let $\nu\in\C$, $\mu\in(-\frac12,\infty)\setminus\{0\},$
$z\in\C\setminus(-\infty,1]$ and $x\in\C$ lies inside the ellipse with foci at $\pm 1$ that passes through $z$. Using 
\eqref{pGf}
and \eqref{Whip1}, one has
\begin{equation}
(z-x)^\nu=2^\mu\Gamma(\mu)(z^2-1)^{\frac{\nu+\mu}{2}}
\sum_{n=0}^\infty
(n+\mu)(-\nu)_n
P_{\nu+\mu}^{-n-\mu}
\left(\frac{z}{\sqrt{z^2-1}}\right)C_n^\mu(x).
\label{zmxLegendreP}
\end{equation}
Let $p\in\N_0$. One then may use the following identity
\[
(z-x)^p\log(z-x)=\lim_{\nu\to0}\frac{\partial}{\partial\nu}(z-x)^{\nu+p},
\]
which upon substituting $\nu\mapsto\nu+p$ in \eqref{zmxLegendreP} produces
\[
(z-x)^p\log(z-x)=2^\mu(z^2-1)^{\frac{\mu+p}{2}}\sum_{n=0}^\infty (n+\mu)(-1)^nC_n^\mu(x)
\lim_{\nu\to0} 
\frac{\partial}{\partial\nu} 
\frac{
(z^2-1)^{\frac{\nu}{2}}
\Gamma(\nu+p+1)}{\Gamma(\nu+p+2\mu+1)}
P_{\nu+p+\mu}^{n+\mu}
\left(\frac{z}{\sqrt{z^2-1}}\right).
\]
Performing the derivatives, one obtains 
\begin{eqnarray}
&&\hspace{-1cm}\frac{\partial}{\partial\nu} 
\frac{
(z^2-1)^{\frac{\nu}{2}}
\Gamma(\nu+p+1)}{\Gamma(\nu+p+2\mu+1)}
P_{\nu+p+\mu}^{n+\mu}
\left(\frac{z}{\sqrt{z^2-1}}\right)\nonumber\\[0.2cm]
&&= \frac{
(z^2-1)^{\frac{\nu}{2}}
\Gamma(\nu+p+1)}{\Gamma(\nu+p+2\mu+1)}
\left(\log\sqrt{z^2-1}+\psi(p+\nu+1)-\psi(p+n+\nu+2\mu+1)\right)
P_{\nu+p+\mu}^{n+\mu}
\left(\frac{z}{\sqrt{z^2-1}}\right)\nonumber\\[0.2cm]
&&+
\frac{
(z^2-1)^{\frac{\nu}{2}}
\Gamma(\nu+p+1)}{\Gamma(\nu+p+2\mu+1)}
\frac{\partial}{\partial\nu}P_{\nu+p+\mu}^{n+\mu}
\left(\frac{z}{\sqrt{z^2-1}}\right),\nonumber
\end{eqnarray}
and after taking the limit one has
\begin{eqnarray}
&&\hspace{-1cm}\lim_{\nu\to0}\frac{\partial}{\partial\nu} 
\frac{p!}{\Gamma(p+2\mu+1)}
P_{p+\mu}^{n+\mu}
\left(\frac{z}{\sqrt{z^2-1}}\right)\nonumber\\[0.2cm]
&&= \frac{p!}{\Gamma(p+2\mu+1)}
\left(\log\sqrt{z^2-1}+\psi(p+1)-\psi(p+n+2\mu+1)\right)
P_{p+\mu}^{n+\mu}
\left(\frac{z}{\sqrt{z^2-1}}\right)\nonumber\\[0.2cm]
&&+
\frac{p!}{\Gamma(p+2\mu+1)}
\left[\frac{\partial}{\partial\nu}P_{\nu}^{n+\mu}
\left(\frac{z}{\sqrt{z^2-1}}\right)\right]_{\nu=p+\mu}.\nonumber
\end{eqnarray}

\noindent For general $\mu$, the parameter derivative of the associated Legendre function
of the first kind given in the above equation is not known. However for $\mu\in\N_0$ it is known (see 
\eqref{limitderivdegreepnum}).
The expansion for the logarithmic kernel
for $\mu=0$ corresponds to the Chebyshev
polynomial of the first kind (see \eqref{limitGegChebyI}), and therefore corresponds to 
\eqref{limitderivCheby}. Hence from this point forward we treat $\mu\in\N$,
and the result follows using \eqref{limitderivdegreeqnum}. 
\end{proof}

%
%

One interesting consequence of \eqref{limitderivGeg} is the following expansion.
\begin{cor}
Let $m\in\N$, {$z\in\C\setminus(-\infty,1]$,
$z-x\in\C\setminus(-\infty,0]$, $x$ lies inside the ellipse with foci at $\pm 1$ that passes through $z$.}
Then
\begin{eqnarray}
&&\hspace{-0.5cm}\log(z-x)=\log\frac{z+\sqrt{z^2-1}}{2}+H_{2m}-H_m\nonumber\\
&&\hspace{2cm}+i(-1)^{m+1}\frac{2^{m+\frac12}}{\sqrt{\pi}}m!(z^2-1)^{\frac{m}{2}+\frac14}
\sum_{k=0}^{m-1}\frac{2k+1}{(k+m)!(m-k)(k+m+1)}Q_{m-\frac12}^{k+\frac12}(z)
\nonumber
\\[0.2cm]
&&\hspace{2cm}+i(-1)^{m}\frac{2^{m+\frac12}}{\sqrt{\pi}}(m-1)!(z^2-1)^{\frac{m}{2}+\frac14}
\sum_{n=1}^{\infty}\frac{(n+m)n!C_n^m(x)}{(2m+n)!}Q_{n+m-\frac12}^{m+\frac12}(z).
\label{zmxgeg}
\end{eqnarray}
\end{cor}
\begin{proof}
Let $p=0$ in \eqref{limitderivGeg} using the duplication theorem \cite[(5.5.5)]{NIST:DLMF}, and
\eqref{Qmhmph}, completes the proof.
\end{proof}

\begin{cor}
\noindent Let 
{$z\in\C\setminus(-\infty,1]$,
$z-x\in\C\setminus(-\infty,0]$, $x$ lies inside the ellipse with foci at $\pm 1$ that passes through $z$.}
Then
\[
\log(z-x)=\log\frac{z+\sqrt{z^2-1}}{2}+\frac12-\frac{\sqrt{z^2-1}}{z+\sqrt{z^2-1}}-2\sum_{n=1}^\infty 
\frac{U_n(x)}{n(n+2)}
\frac{z+(n+1)\sqrt{z^2-1}}{(z+\sqrt{z^2-1})^{n+1}}.
\]
\end{cor}
\begin{proof}
Let $\mu=1$ in  \eqref{zmxgeg}, then use \eqref{Qhh}, \eqref{Qnhnp1}.
Simplification completes the proof.
\end{proof}

\section{Azimuthal Fourier and Gegenbauer
polynomial 
expansions of binomial
and logarithmic fundamental solutions}
\label{FundExpansions}

{
The behavior of a logarithmic fundamental
solution of the polyharmonic equation on
even-dimensional Euclidean space $\R^d$ in 
a rotationally-invariant coordinate system
and in a Vilenkin polyspherical coordinate
system are given respectively by \eqref{lft1}, \eqref{lft2}. 
For rotationally-invariant
coordinate systems \eqref{rotatioanallyinvariant}, recall that $R,R'$ are the 
cylindrical radii and $\phi$, $\phi'$ are azimuthal angles  and that
$\chi$, the toroidal parameter, is defined 
in \eqref{chidefn}.
Furthermore, for the 
polyharmonic equation \eqref{polyharmoniceq}, 
$k\in\N$ is the 
power of the positive Laplacian.
For the definitions of the special functions
and numbers used in the results presented
in this section, see Section 
\ref{Specialfunctionsandorthogonalpolynomials}.
We now present the expression for the azimuthal
Fourier expansion of a logarithmic fundamental
solution of the polyharmonic equation.
}


{
\begin{thm}
Let $\bfx,\bfxp\in\R^d$ with $d\in 2\N$,
$p=k-d/2\in\N_0$.
Then the azimuthal Fourier expansion of
a logarithmic fundamental solution
of the polyharmonic equation $\li_k^d$ in a 
rotationally-invariant coordinate system on 
Euclidean space $\R^d$ is given by
\begin{eqnarray}\label{fourierlog}
&&\hspace{-1.1cm}\li_k^d(\bfx,\bfxp)=
\frac{ip!}{\sqrt{2\pi}}\left(2RR^\prime\right)^p
(\chi^2-1)^{\frac{p}{2}+\frac14}\Biggl[\left(\log
\left(RR^\prime\right) + \log (\chi+\sqrt{\chi^2-1})+2H_{2p}- H_p - H_{\frac{d}{2}+p-1}+H_{\frac{d}{2}-1}\right) \nonumber \\
&& \hspace{1cm}\times (-1)^{p+1}
\sum_{m=0}^p\frac{\epsilon_m(-1)^m \cos{\left(m (\phi-\phi')\right)}}{(p-m)!(p+m)!}Q_{m-\frac12}^{p+\frac12}(\chi)  \nonumber \\
&&\hspace{+0.20cm}+(-1)^p
\sum_{m=0}^p
\frac{\epsilon_m(-1)^m \cos{\left(m (\phi-\phi')\right)}}{(p-m)!(p+m)!}
\left(
H_{p+m}+H_{p-m}
\right)
Q_{m-\frac12}^{p+\frac12}(\chi)
\nonumber\\[0.2cm]
&&\hspace{+0.20cm}
+(-1)^{p+1}\sum_{m=0}^{p-1}
\frac{\epsilon_m(-1)^m \cos{\left(m (\phi-\phi')\right)}}{(p+m)!} 
\sum_{k=0}^{p-m-1}\frac
{(2m+2k+1)\left[1+\frac{k!(p+m)!}{(2m+k)!(p-m)!}\right]}
{k!(p-m-k)(p+m+k+1)}
Q_{m-\frac12}^{m+k+\frac12}(\chi)
\nonumber\\[0.2cm]
&&\hspace{+0.20cm} 
+2
(-1)^{p+1}\sum_{m=1}^{p}\frac{(-1)^m \cos{\left(m (\phi-\phi')\right)}}{(p-m)!}
\sum_{k=0}^{m-1}\frac
{2k+1}
{(m+k)!(p-k)(p+k+1)}
Q_{m-\frac12}^{k+\frac12}(\chi)
\nonumber\\[0.2cm]
&&\hspace{+0.20cm} 
+2
\sum_{m=p+1}^\infty \frac{(m-p-1)!\,\cos{\left(m (\phi-\phi')\right)}}{(p+m)!}
Q_{m-\frac12}^{p+\frac12}(\chi) \Biggr].
\end{eqnarray}
\end{thm}
}
\begin{proof}
Beginning with \eqref{lft1}, applying the identities \eqref{Eulerpolyexp}, \eqref{limitderivCheby} 
and simplifying, completes the proof.
\end{proof}

{A logarithmic fundamental solution of the polyharmonic
equation expressed in a Vilenkin polyspherical coordinate system is given by \eqref{lft2}.
 Also recall that
$r$, $r'$ are the hyperspherical radii and 
$\cos\gamma=(\bfx,\bfxp)/(\|\bfx\|\|\bfxp\|)$ 
\eqref{sepang} is the separation
 angle \eqref{sepang}
 in a Vilenkin 
 polyspherical coordinate system \cite[Appendix B and references therein]{Cohl12pow}, and that 
$\zeta=(r^2+r'^2)/(2rr')$ is defined 
in \eqref{zetadef}.
We now give the Gegenbauer expansion for a logarithmic fundamental solution of the polyharmonic equation in Vilenkin polyspherical coordinate systems.}

\begin{thm} \label{gegexpand}
Let $\bfx,\bfxp\in\R^d$ with $d\in2\N$,
$p=k-d/2\in\N_0$.
Then the Gegenbauer expansion of
a logarithmic fundamental solution
of the polyharmonic equation $\li_k^d$ on 
even-dimensional 
Euclidean space $\R^d$ in a Vilenkin polyspherical
coordinate system is given by
\begin{eqnarray}
&&\hspace{-0.7cm}\li_k^d(\bfx,\bfxp)= i
\frac{2^{\mu-\frac12}}{\sqrt{\pi}} p! \Gamma(\mu) (\zeta^2-1)^{\frac{p+\mu}{2}+\frac14} \left(2rr^\prime\right)^p \Biggl[ \left(\log
\left(rr^\prime\right) + \log{\left(\zeta+\sqrt{\zeta^2-1}\right)}+2H_{2p+2\mu}+H_p-H_{p+\mu}-2\beta_{p,d}\right) \nonumber \\
&& \hspace{-0.5cm}\hspace{1.2cm} \times (-1)^{p+\mu+1}
\sum_{l=0}^p \frac{(-1)^l(l+\mu)C_l^\mu(\cos \gamma)}
{(p-l)!(l+p+2\mu)!}
Q_{l+\mu-\frac12}^{p+\mu+\frac12}(\zeta). \nonumber \\
&&\hspace{-0.5cm}\hspace{+0.20cm}+(-1)^{p+\mu}\sum_{l=0}^p
\frac{(-1)^l(l+\mu) C_l^\mu(\cos \gamma)}{(p-l)!(p+l+2\mu)!}
\left(
H_{p+l+2\mu}+H_{p-l}
\right)
Q_{l+\mu-\frac12}^{p+\mu+\frac12}(\zeta)
\nonumber\\[0.2cm]
&&\hspace{-0.5cm}\hspace{+0.20cm}
+(-1)^{p+\mu+1}
\sum_{l=0}^{p-1}
\frac{(-1)^l(l+\mu) C_l^\mu(\cos \gamma)}{(p+l+2\mu)!}\nonumber\\[-0.0cm]
&&\hspace{-0.5cm}\hspace{1.2cm}\times\sum_{k=0}^{p-l-1}\frac
{(2l+2k+2\mu+1)
}
{k!(p-l-k)(p+l+k+2\mu+1)}
\left[1+\frac{k!(p+l+2\mu)!}{(k+2l+2\mu)!(p-l)!}\right]
Q_{l+\mu-\frac12}^{k+l+\mu+\frac12}(\zeta)
\nonumber\\[-0.0cm]
&&\hspace{-0.5cm}\hspace{+0.20cm}
+ (-1)^{p+\mu+1}
\sum_{l=0}^{p}\frac{(-1)^l(l+\mu)C_l^\mu(\cos \gamma)}{(p-l)!}
\sum_{k=0}^{l+\mu-1}\frac
{2k+1}
{(l+k+\mu)!(p+\mu-k)(p+k+\mu+1)}
Q_{l+\mu-\frac12}^{k+\frac12}(\zeta)
\nonumber\\[0.2cm]
&&\hspace{-0.5cm}\hspace{+0.20cm}
+(-1)^\mu
\sum_{l=p+1}^\infty \frac{(l+\mu)(l-p-1)!C_l^\mu(\cos \gamma)}
{(p+l+2\mu)!}
Q_{l+\mu-\frac12}^{p+\mu+\frac12}(\zeta) \Biggr].
\end{eqnarray}
\end{thm}
\begin{proof}
Applying the identities \eqref{limitderivGeg} and \eqref{binomGeg} to \eqref{lft2} completes the proof.
\end{proof}

{For the polyharmonic equation on even-dimensional Euclidean space $\R^d$ with 
$k \ge \frac{d}{2}$, apart from multiplicative constants, the algebraic 
expression for a binomial fundamental solution of the polyharmonic
equation $\jii_k^d:\R^d\times\R^d\to\R$ 
is given by
$
\jii_k^d(\bfx,\bfxp):=\|\bfx-\bfxp\|^{2k-d},
$
with $2k-d\in2\N$.
By expressing $\jii_k^d$ in a rotationally-invariant coordinate system one has \eqref{unlogfourseries}
which leads to the following result for an azimuthal
Fourier expansion of the binomial fundamental
solution of the polyharmonic equation.
}
{
\begin{thm} \label{bifourier}
Let $\bfx,\bfxp\in\R^d$ with $d\in2\N$, $p=k-d/2\in\N_0$.
Then the azimuthal Fourier expansion of
a binomial fundamental solution
of the polyharmonic equation $\jii_k^d$ on 
Euclidean space $\R^d$ is given by
\begin{eqnarray} 
&&\jii_k^d(\bfx, \bfxp)=i(-1)^{p+1}\sqrt{\frac{2}{\pi}}p!(2RR')^{p}{(\chi^2-1)^{\frac{p}{2}+\frac14}}
\sum_{n=0}^p\frac{(-1)^n\epsilon_n \cos\left( n(\phi - \phi')\right)}{(p-n)!(p+n)!}Q_{n-\frac12}^{p+\frac12}(\chi).
\end{eqnarray}
\end{thm}
}
{
\begin{proof}
Starting with \eqref{unlogfourseries}, and using the binomial expansion \eqref{Eulerpolyexp} with the Fourier cosine representation of the Chebyshev polynomials of the first kind completes
the proof.
\end{proof}
}


{
The expression $\jii_k^d$ can also be represented in any Vilenkin polyspherical coordinate system \eqref{unlogVilen} which likewise can be expanded using Gegenbauer polynomials. This is presented as in the following result.}

{
\begin{thm}\label{bigeg}
Let $\bfx,\bfxp\in\R^d$ with $d\in2\N$, 
$p=k-d/2\in\N_0$.
Then the Gegenbauer expansion of
a binomial fundamental solution
of the polyharmonic equation $\jii_k^d$ on 
Euclidean space $\R^d$ is given by
\begin{eqnarray} 
&&\hspace{-0.5cm}\jii_k^d(\bfx,\bfxp)=i
(-1)^{p-\frac{d}{2}}\frac{2^{\frac{d-1}{2}}}{\sqrt{\pi}}
p!(\tfrac{d}{2}-2)! (2rr')^p 
\pdh \nonumber \\
&&\hspace{6cm}\times \sum_{n=0}^p \frac{(-1)^n(n+\frac{d}{2}-1)C_n^{\frac{d}{2}-1}(\cos\gamma)}
{(p-n)!(n+p+d-2)!}
Q_{n+\frac{d-3}{2}}^{p+\frac{d-1}{2}}(\zeta).
\end{eqnarray}
\end{thm}
}
{
\begin{proof}
The Gegenbauer polynomial expansion of Lemma \ref{binomGeg} with $\mu = d/2-1$, is combined with a fundamental solution written in a Vilenkin polyspherical coordinate system and using
\eqref{zetasquare} completes the proof.
\end{proof}
}


\section{Global analysis 
on $\R^d$ of standard and generalized Hopf Vilenkin polyspherical coordinate systems}
\label{secstangenHopf}

\medskip
{Now we study some of the particular details which 
will arise in the study of binomial and logarithmic fundamental solutions of the polyharmonic equation
in standard polyspherical coordinates and generalized Hopf coordinates. We will utilize the addition
theorem for hyperspherical harmonics to expand
the critical Gegenbauer polynomial (that with an order equal to $d/2-1$) over the product of normalized hyperspherical harmonics in that particular Vilenkin
polyspherical coordinate system.
The reason we used the critical order $d/2-1$ is that
these Gegenbauer polynomials with this particular order provide a basis for analytic solutions on $\Si_r^{d-1}$.} 
For $d\ge 3$, the addition theorem for hyperspherical harmonics is given by (for a proof see \cite{WenAvery}, \cite[\S 10.2.1]{FanoRau})
\begin{equation}
C_n^{d/2-1}(\cos\gamma)=\frac{2(d-2)\pi^{d/2}}
{(2n+d-2)\Gamma(d/2)}\sum_{K}Y_n^K(\wbfx)\overline{Y_n^K(\wbfxp)},
\label{addthmhypsph}
\end{equation}
where $K$ stands for a set of $(d-2)$-quantum numbers
identifying harmonics
for a given value of ${n}\in\N_0$, and $\cos\gamma$ is the cosine 
of the separation angle \eqref{sepang} between two arbitrary
vectors $\bfx,\bfxp\in\R^d$.
The functions
$Y_{{n}}^{K}:\Si^{d-1}\to\C$ are the normalized hyperspherical harmonics.
Normalization of the hyperspherical harmonics is achieved through the integral
\[
\int_{\Si^{d-1}}Y_n^K(\wbfx)\overline{Y_n^K(\wbfx)}d\Omega=1,
\]
where $d\Omega$ is the Riemannian volume measure on $\Si^{d-1}$.


First we will treat standard polyspherical coordinates and then we will treat generalized Hopf coordinates.
Both of these polyspherical coordinate systems and many of their various properties are described in 
\cite[Appendix B]{Cohl12pow}.

\subsection{Standard polyspherical coordinates}

{Here we review details connected with 
standard Vilenkin polyspherical coordinates. These coordinates 
and as well the general Vilenkin polyspherical coordinates are described carefully in \cite[Appendix B]{Cohl12pow}, and we will not depart from the description and usage of standard polyspherical coordinates described therein.}
Standard polyspherical coordinates are a generalization of the spherical coordinate system that is most commonly encountered in multi-dimensional calculus.
{
What we refer to as {\it standard polyspherical coordinates}
are given by
\begin{gather}
x_{1} = r\cos\theta_1,\nonumber\\
x_{2} = r\sin\theta_1\cos\theta_2,\nonumber\\
x_{3} = r\sin\theta_1\sin\theta_2\cos\theta_3,\nonumber\\
\cdots \cdots\cdots\cdots\cdots\cdots\cdots\cdots\nonumber\\
x_{d-2} = r\sin\theta_1\cdots\sin\theta_{d-3}\cos\theta_{d-2},\nonumber\\
x_{d-1} = r\sin\theta_1\cdots\sin\theta_{d-3}\sin\theta_{d-2}\cos\phi,\nonumber\\
x_{d} = r\sin\theta_1\cdots\sin\theta_{d-3}\sin\theta_{d-2}\sin\phi,\label{standardsph}
\end{gather}
where $\theta_i\in[0,\pi]$ for
$1\le i \le d-2$ and
$\phi\in[-\pi,\pi)$.}
In standard polyspherical coordinates, the
normalized hyperspherical harmonics can be written as \cite[(B.19)]{Cohl12pow}
\begin{equation}\label{normhyper}
    Y_l^K(\hat{{\bf x}}) = \frac{ \expe^{ i m \phi}}{\sqrt{2 \pi}} \prod_{j=1}^{d-2} \Theta_j^d(l_j,l_{j+1};\theta_j),
\end{equation}
where  \cite[(B.20)]{Cohl12pow}
\begin{align} \label{thetaGeg}
\Theta_j^d (l_j,l_{j+1};\theta_j) :=& \frac{ \Gamma \left( l_{j+1} + \frac{d-j+1}{2} \right)}{2l_{j+1}+d-j-1} \sqrt{ \frac{2^{2l_{j+1}+d-j-1} (2 l_j+d-j-1)(l_j-l_{j+1})!}{ \pi (l_j+l_{j+1} +d-j-2)!}} \nonumber \\
& \times  (\sin \theta_j)^{l_{j+1}} C_{l_j-l_{j+1}}^{l_{j+1} + (d-j-1)/2}( \cos \theta_j).
\end{align}
The addition theorem 
for hyperspherical harmonics \eqref{addthmhypsph} involves the product 
$Y_l^K(\bfx)\overline{Y_l^K(\bfxp)}$, so we introduce a convenient notation, namely 
$\Omega_k^d:\N_0\times\Z\times[0,\pi]^2\to\R$ is defined by 
\begin{equation} \label{standpoly}
\OmeH{k}{l_{k+1}}:=
\Theta_k^d(l_k,l_{k+1};\theta_k)
\Theta_k^d(l_k,l_{k+1};\theta_k').
\end{equation}

\subsubsection{Standard polyspherical coordinates multi-sum reversal lemmas}

{
An important ingredient in the production of the binomial and logarithmic addition functions is reversing the order of the multi-sums that appear when the Gegenbauer polynomials are expanded as multi-sums of hyperspherical harmonics.  This allows us to compare the Fourier coefficients relying on the azimuthal quantum number $m$.
}
{
\begin{lem} \label{multisum}
Let $p \in \mathbb{N}_0$. Then the multi-sum $\sf Y_1$, over the allowed quantum numbers for standard polyspherical hyperspherical harmonics
\eqref{normhyper}
defined by
\begin{align} 
    {\sf Y}_1 := \sum_{l=0}^p \sum_K &= \sum_{l=0}^p \sum_{l_2=0}^l \cdots \sum_{l_{d-2}=0}^{l_{d-3}} \sum_{m=0}^{l_{d-2}},
\end{align}
can then be re-expressed with the sum order reversed as
\begin{align}
    {\sf Y}_1= \sum_{m=0}^p \sum_{l_{d-2}=m}^p \cdots \sum_{l_2=l_3}^p \sum_{l=l_2}^p.
\end{align}
\end{lem}
}
\begin{proof}
To reverse the multi-sum, the upper and lower bounds of the indices need to be determined.  Each index of the original multi-sum can have values between 0 and $p$ inclusively.  However, the upper bounds of the original multi-sum gives the following constraint:
\begin{align} \label{sumcon}
    0 \le m \le l_{d-2} \le ... \le l_2 \le l.
\end{align}
When the indices are reversed, this constraint determines the new lower bound of each of the indices.  Since there are no other constraints, the upper bound for each index will be $p$.
\end{proof}
{
\begin{lem} \label{lp1sumlem}
Let $p \in \mathbb{N}_0$. Then the multi-sum ${\sf Y}_2$
over quantum numbers for standard polyspherical harmonics defined by 
\begin{align*} 
    {\sf Y}_2 &:=\sum_{l=p+1}^{\infty} \sum_K  = \sum_{l=p+1}^{\infty} \sum_{l_2=0}^l \cdots \sum_{l_{d-2}=0}^{l_{d-3}} \sum_{m=0}^{l_2}.
\end{align*}
can then be re-expressed with the sum order reversed as
\begin{align}
    {\sf Y}_2  = \sum_{m=0}^p \sum_{l_{d-2}=m}^{\infty} \cdots \sum_{l_2=l_3}^{\infty} \sum_{l=\max(l_2,p+1)}^{\infty} 
     + \sum_{m=p+1}^{\infty} \sum_{l_{d-2}=m}^{\infty} \cdots \sum_{l_2=l_3}^{\infty} \sum_{l=l_2}^{\infty}.
\end{align}
\end{lem}
}
\begin{proof}
First it is noted that each of the indices in the original multi-sum, except for $l$, will have a lower bound of 0 and an upper bound of infinity. The index $l$ will have the same upper bound, but a lower bound of $p+1$.  Like the previous lemma, the upper bounds of the original multi-sum will give us the constraint, \eqref{sumcon}, which determines the lower bounds of the reversed multi-sum. The combination of the following two lower bounds:
\begin{align*}
    l_2 \le l, \quad
    p+1 \le l,
\end{align*}
yields the lower bound $\max(p+1, l_2)$.  With this, the multi-sum can be reversed:
\begin{align*}
    {\sf Y_2} = \sum_{m=0}^\infty \sum_{l_{d-2}=m}^\infty \cdots \sum_{l_2=l_3}^\infty \sum_{l =\text{max}(l_2,p+1)}^\infty.
\end{align*}
In later theorems, the scenarios of $m \le p$ and $m \ge p+1$ need to be treated separately.  With the index $m$ being the first sum of the multi-sum, this split is uncomplicated.  When $m \ge p+1$, we can drop the maximum function on the lower bound of $l$.
\end{proof}

\subsection{Generalized Hopf coordinates}
\label{Ghopfsec}

{
Generalized Hopf
coordinates are a type of Vilenkin
polyspherical coordinates
on $\R^d$ with $d=2^q$ for $q\ge 1$.
They are Vilenkin polyspherical
orthogonal curvilinear 
coordinates with one radial coordinate $r\in[0,\infty)$, and $(d-1)$-angular coordinates which together parametrize points on $\Si_r^{d-1}$ the $(d-1)$-dimensional real hypersphere with radius $r$. 
Of the $(d-1)$-angular coordinates
$(d/2\!-\!1)$-$\vartheta$ coordinates take values in $[0,\frac{1}{2}\pi]$,
and the other $(d/2)$-$\phi$ coordinates are of azimuthal type
and take values in $[-\pi,\pi)$.
For a careful treatment of generalized Hopf coordinates, see \cite[Appendix B]{Cohl12pow}.
}

{
In this paper we depart slightly from
our previous description of generalized Hopf coordinates---we have adopted a {\it reversed} azimuthal identification
for the azimuthal angles and their corresponding quantum numbers.
In particular if one considers the collection
of angles in generalized Hopf coordinates given by some vector of angles ${\bf \boldsymbol\Theta}:=(\Theta_1,\ldots,\Theta_{d-1})$ with vector of corresponding quantum numbers ${\bf p}:=(p_1,\ldots,p_{d-1})$. In our previous paper, these were ordered as
\begin{eqnarray}
&&\boldsymbol\Theta=(\vartheta_1,\ldots,\vartheta_{d/2-1},\phi_1,\ldots,\phi_{d/2}), \quad
{\bf p}=(l_1,\ldots,l_{d/2-1},m_1,\ldots,m_{d/2}),
\end{eqnarray}
and in the current paper we order them as 
\begin{eqnarray}
&&\boldsymbol\Theta=(\vartheta_1,\ldots,\vartheta_{d/2-1},\phi_{d/2},\ldots,\phi_1), \quad
{\bf p}=(l_1,\ldots,l_{d/2-1},m_{d/2},\ldots,m_1)
\end{eqnarray}
(see Figure~\ref{Fig:genHopf} as compared to \cite[Figure 5]{Cohl12pow}).
}

{
These coordinates
generalize two-dimensional polar coordinates (see Figure \ref{Fig:genHopf}a)
\begin{gather}
x_1 = r\cos\phi, \qquad
x_2 = r\sin\phi,
\label{hypsph2}
\end{gather}
and four-dimensional Hopf coordinates
(see Figure \ref{Fig:genHopf}b)
\begin{alignat}{3}
& x_1 = r\cos\vartheta\cos\phi_2
,\qquad &&
x_2 = r\cos\vartheta\sin\phi_2, &\nonumber\\
& x_3 = r\sin\vartheta\cos\phi_1,
\qquad &&
x_4 = r\sin\vartheta\sin\phi_1.
& \label{typeca2}
\end{alignat}
See Figure \ref{Fig:genHopf}c for a Vilenkin tree of the $d=8$ generalized Hopf coordinates.
In general the transformation formulae to 
Cartesian coordinates for generalized Hopf coordinates is given by
\begin{gather}
x_{1} = r\cos\vartheta_1\cos\vartheta_2\cos\vartheta_4\cos\vartheta_8\cdots\cos\vartheta_{2^{q-2}}
\cos\phi_{2^{q-1}}
,\nonumber\\
x_{2} = r\cos\vartheta_1\cos\vartheta_2\cos\vartheta_4\cos\vartheta_8\cdots\cos\vartheta_{2^{q-2}}
\sin\phi_{2^{q-1}}
,\nonumber\\
\cdots\cdots\cdots\cdots\cdots\cdots\cdots\cdots\cdots\cdots\cdots\cdots\cdots\cdots\cdots\cdots  \nonumber\\
x_{2^{q-1}-1} = r\cos\vartheta_1\sin\vartheta_2\sin\vartheta_5\sin\vartheta_{11}\cdots
\sin\vartheta_{3\cdot 2^{q-3}-1}
\cos\phi_{2^{q-2}+1}
,\nonumber\\
x_{2^{q-1}} = r\cos\vartheta_1\sin\vartheta_2\sin\vartheta_5\sin\vartheta_{11}\cdots
\sin\vartheta_{3\cdot 2^{q-3}-1}
\sin\phi_{2^{q-2}+1}
,\nonumber\\
x_{2^{q-1}+1} = r\sin\vartheta_1\cos\vartheta_3\cos\vartheta_6\cos\vartheta_{12}\cdots
\cos\vartheta_{3\cdot 2^{q-3}}
\cos\phi_{2^{q-2}}
,\nonumber\\
x_{2^{q-1}+2} = r\sin\vartheta_1\cos\vartheta_3\cos\vartheta_6\cos\vartheta_{12}\cdots
\cos\vartheta_{3\cdot 2^{q-3}}
\sin\phi_{2^{q-2}}
,\nonumber\\
\cdots\cdots\cdots\cdots\cdots\cdots\cdots\cdots\cdots\cdots\cdots\cdots\cdots\cdots\cdots\cdots  \nonumber\\
x_{2^{q}-1} = r\sin\vartheta_1\sin\vartheta_3\sin\vartheta_7\sin\vartheta_{17}\cdots
\sin\vartheta_{2^{q-1}-1}
\cos\phi_1
,\nonumber\\
x_{2^{q}} = r\sin\vartheta_1\sin\vartheta_3\sin\vartheta_7\sin\vartheta_{17}\cdots
\sin\vartheta_{2^{q-1}-1}
\sin\phi_1
,
\label{genHopf}
\end{gather}
where $\vartheta_j\in\big[0,\frac12\pi\big]$ for
$1\le j \le 2^{q-1}\!-\!1$
and $\phi_k\in[-\pi,\pi)$ for
$1\le k\le 2^{q-1}$.
Generalized Hopf coordinates are unique in that they correspond
to the only trees which contain only themselves
in their equivalence class (see~\cite[(B.2)]{Cohl12pow}).
These coordinate systems have separated harmonic eigenfunctions which
are given in terms of complex
exponentials of the azimuthal angles, and for $q\ge 2$, non-symmetric Jacobi polynomials for 
the quantum numbers corresponding to the $\vartheta$-angles.
}

\begin{figure}[H]
\hspace{0.8cm}\includegraphics[scale=0.79,angle=90,origin=c,trim={0.3cm 0.0cm 0.0cm 1.9cm},clip]{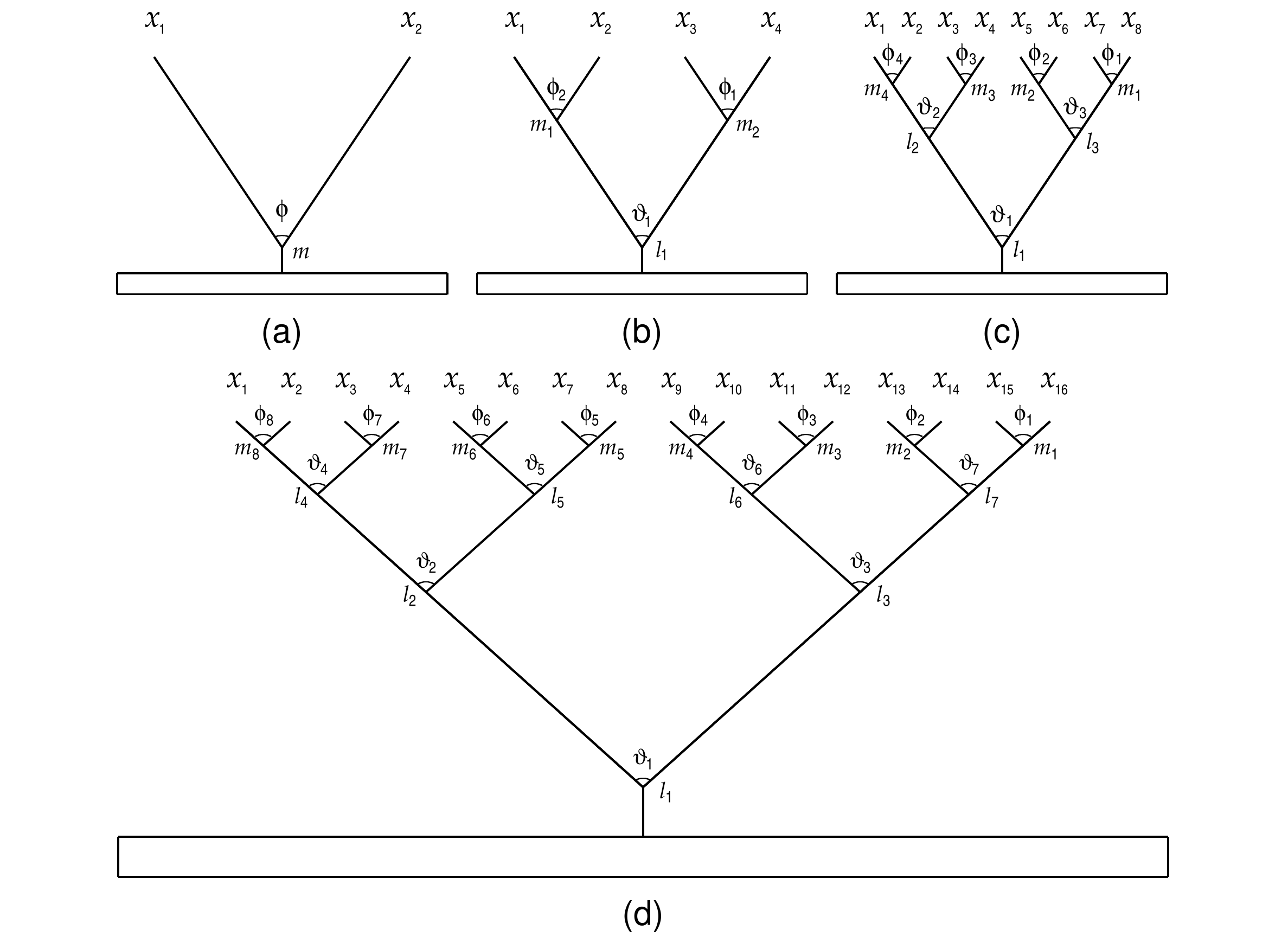}\\[-1.5cm]
\caption{{This figure is a tree diagram for polyspherical generalized Hopf
coordinates of 
on~$\R^{2^{q}}$ with $q=1,2,3,4$
for (a), (b), (c), (d) respectively.
The first $(2^{q-1}\!-\!1)$-branching nodes  correspond to the angles
$\vartheta_j\in\big[0,\frac12\pi\big]$
and quantum numbers~$l_j\in\N_0$, $1\le j\le d/2\!-\!1$.  The following $(2^{q-1})$-branching nodes 
correspond to the angles $\phi_k\in[-\pi,\pi)$ and quantum numbers $m_k\in\Z$, $1\le k\le d/2$.
These coordinates correspond to transformation~(\ref{genHopf}).}}
\label{Fig:genHopf}
\end{figure}


{
The cosine of the separation angle 
\eqref{sepang}
in these coordinates may be given as follows.
Define the symbol ${}_q{\sf G}_s^r\in[-1,1]$, where $0\le s\le q$
and $1\le r \le 2^{q}-1$,
by the recursive formula
\[
 {}_{q}{\sf G}_s^r=
\cos\vartheta_{r-1+2^{q-s}}
\cos\vartheta_{r-1+2^{q-s}}'
\ {}_{q}{\sf G}_{s-1}^{2r-1}
+
\sin\vartheta_{r-1+2^{q-s}}
\sin\vartheta_{r-1+2^{q-s}}'
\ {}_{q}{\sf G}_{s-1}^{2r},
\]
with ${}_{q}{\sf G}_0^i=1$. Then the cosine of the separation angle is
given by
\[
\cos\gamma={}_{q}{\sf G}_q^1.
\]
Note that through the identification
$\phi_k=\Theta_{d/2+1-k}$, where
$1\le k\le 2^{q-1}$,
then
${}_{q}{\sf G}_1^i=\cos(\phi_i-\phi_i')$. Thus, this shows one may stop this recursion at $s=1$.
}

{
In regard to the quantum numbers in 
generalized Hopf coordinates, denote the
meridional quantum numbers $l_k\in\N_0$
and the azimuthal quantum numbers $m_j\in\Z$, 
such that $1\le k\le d/2-1$, $1\le j\le d/2$. 
Define $\mcL:=(l_1,l_2,\ldots,l_{\fdt-1})$,  $\mcM:=(m_\fdt,\ldots,m_1)$ where $l:=l_1$ 
and $m:=m_1$.
In these coordinates, it is convenient to relate the meridional quantum numbers $l_k$ to corresponding
{\it surrogate}
quantum numbers $n_k$ using $l_k=2n_k+l_{2k}+l_{2k+1}$, and we define
$\mcN:=(n_1,n_2,\ldots,n_{\fdt-1})$, where $n:=n_1$, and $n_k\in\N_0$ for all $1\le k\le d/2-1$.
For convenience define $N:=\sum_{k=1}^{d/2-1}n_k$, 
$M:=\sum_{j=1}^{d/2} m_j$.
Note that one can always write
\begin{eqnarray}
&&l_1=2N+M,\nonumber\\[0.2cm]
&&l_2=
2\sum_{j=2}^{\log_2(d/2)}\sum_{k=0}^{2^{j-2}-1}n_{2^{j-1}+k}+
\sum_{j=\fdf+1}^\fdt m_j,\nonumber\\[0.2cm]
&&l_3=
2\sum_{j=2}^{\log_2(d/2)}\sum_{k=0}^{2^{j-2}-1}n_{2^{j-2}+2^{j-1}+k}
+
\sum_{j=1}^\fdf m_j.\nonumber
\end{eqnarray}
}

In generalized Hopf coordinates, the normalized hyperspherical harmonics are given by 
\cite[(B.21)]{Cohl12pow}
\[
Y_l^K(\bfx)=\frac{\prod_{1\le j\le d/2}\exp(im_j\phi_j)}{\sqrt{2}\,\pi^{d/4}}
\upSH{1}{\log_2d}{l_2,l_3}
\times\cdots\times
\upsH{\fdt-1}{\log_2d}{m_2,m},
\]
where $\Upsilon_{k}^{\log_2d}:\N_0^3\times[0,\frac12\pi]\to\R$ is defined by 
\begin{eqnarray} \label{UpsDef}
&&\upsH{k}{\log_2d}{l_{2k},\l_{2k+1}}:=\sqrt{\frac{(2n_k+\alpha+\beta+1)(n_k+\alpha+\beta)!n_k!}
{(n_k+\alpha)!(n_k+\beta)!}}
\nonumber\\[0.2cm]&&\hspace{5cm}\times
(\cos\vartheta_k)^{l_{2k}}
(\sin\vartheta_k)^{l_{2k+1}}
P_n^{(\beta,\alpha)}(\cos(2\vartheta_k)),
\end{eqnarray}
\begin{eqnarray}
&& \alpha=\alpha_k^d(l_{2k}):=l_{2k}-1+2^{\log_2(d/4)-\lfloor\log_2k\rfloor},\nonumber\\
&& \beta=\beta_k^d(l_{2k+1}):=l_{2k+1}-1+2^{\log_2(d/4)-\lfloor\log_2k\rfloor}.\nonumber
\end{eqnarray}
\begin{remark}
Note that if $d/4\le k\le d/2-1$ then $-1+2^{\log_2(d/4)-\lfloor\log_2k\rfloor}=0$, hence
$\alpha,\beta,l_{2k},l_{2k+1}\in\mcM$. Also if $1\le k\le d/4-1$, 
then $l_{2k},l_{2k+1}\in\mcL$.
\label{upperMremark}
\end{remark}

The addition theorem 
for hyperspherical harmonics \eqref{addthmhypsph} involves the product 
$Y_l^K(\bfx)\overline{Y_l^K(\bfxp)}$, so we introduce a convenient notation, namely 
$\Psi_{k}^{\log_2d}:\N_0^3\times[0,\frac12\pi]^2\to\R$ is defined by 
\begin{equation} \label{psi2ups}
\PsiH{k}{\log_2d}{l_{2k},\l_{2k+1}}:=
\upsH{k}{\log_2d}{l_{2k},\l_{2k+1}}
\upsHp{k}{\log_2d}{l_{2k},\l_{2k+1}}.    
\end{equation}

\begin{thm} \label{hopfaddthm}
In generalized Hopf coordinates, the product of normalized harmonics which appears in 
the addition theorem for hyperspherical harmonics \eqref{addthmhypsph} is given by 
\begin{eqnarray}
&& Y_l^K(\bfx)\overline{Y_l^K(\bfxp)}=\frac{1}{2\pi^{\frac{d}{2}}}
\epsilon_m\cos(m(\phi-\phi'))
\epsilon_{m_2}\cos(m_2(\phi_{m_2}-\phi_{m_2}'))\cdots
\epsilon_{m_\fdt}\cos(m_\fdt(\phi_\fdt-\phi_\fdt'))\nonumber\\
&&
\hspace{3cm}\times\PsiH{\fdt-1}{\log_2d}{m_2,m}
\times\cdots\times
\PsIH{1}{\log_2d}{l_2,l_3}.
\end{eqnarray}
\end{thm}
\begin{proof}
In the sequel, we sum over all quantum numbers. 
If the coefficients of a product of Fourier series over $m_j\in\mcM$, whose coefficients are $f_{m_j}$, are invariant 
under sign reversal transformation, namely $f_{-m_j}=f_{m_j}$,
we can rewrite the product of complex exponentials as a product of trigonometric cosine 
functions. This is accomplished using
\[
\sum_{m_j\in\Z}f_j\exp(im_j\psi_j)=
\sum_{m_j\in\N_0}\epsilon_m f_j \cos(m_j\psi_j),
\]
where $\epsilon_m=2-\delta_{m,0}$.
The eigenfunctions $\Psi_k^{\log_2d}$ are invariant under this transformation, 
which can be verified by applying
\eqref{Jacobinega}, \eqref{Jacobinegb} for $d/4\le k\le d/2-1$ (see Remark \ref{upperMremark}).
\end{proof}

\subsubsection{Generalized Hopf multi-sum reversal lemmas}

{In generalized Hopf coordinates (as well as in the large-part of Vilenkin polyspherical coordinate systems), the 
procedure for summing over the degenerate quantum 
numbers can become somewhat technical. In the 
specialized case of generalized Hopf coordinates
we outline the procedure for describing the multi-sums
over the degenerate quantum numbers and derive 
descriptions of those same multi-sums, but instead
with the sum orders reversed.}

\begin{lem} \label{sumreversal}
Let $p\in\N_0$, 
$1 \le j \le d/2$,
$1 \le k \le d/2-1$.
Consider the multi-sum ${\sf Z}_1$ defined by
\begin{eqnarray}
&&{\sf Z}_1:=\sum_{l=0}^p 
\sum_{l_2=0}^\infty
\cdots
\sum_{l_{2^{q-1}-1}=0}^\infty
\ 
\sum_{m_{2^{q-1}}=0}^\infty
\cdots
\sum_{m_2=0}^\infty
\ 
\sum_{m=0}^p
\end{eqnarray}
over $\mcL\cup\mcM$ with 
$0\le m_j\le p$,
$0\le l_k\le p$, 
where the sum over the $l_k\in\mcL$ quantum numbers are restricted such that 
\[
n_k=\frac{l_k-l_{2k}-l_{2k+1}}{2}\in\N_0.
\]
Then the multi-sum ${\sf Z}_1$ can be re-expressed as a multi-sum over $\mcN\cup\mcM,$ with sum order reversed to obtain
\begin{eqnarray}
&&\hspace{0cm}{\sf Z}_1=
\sum_{m=0}^p
\cdots
\sum_{m_k=0}^{p-{\scriptstyle \sum_{i=1}^{k-1}m_i}}
\cdots
\sum_{m_{2^{q-1}}=0}^{p-{\scriptstyle \sum_{i=1}^{2^{q-1}-1}m_i}}
\ 
\sum_{n_{2^{q-1}-1}=0}^{\left\lfloor \frac{p-{\scriptstyle \sum_{i=1}^{2^{q-1}}m_i}}{2}\right\rfloor}
\cdots
\sum_{n=0}^{\left\lfloor \frac{p-{\scriptstyle \sum_{i=1}^{2^{q-1}}m_i}-2\sum_{i=2}^{2^{q-1}-1}n_i}{2}\right\rfloor}.
\end{eqnarray}
\end{lem}

{
\begin{proof}
Let $q$ denote an integer at least 2.
Put $d = 2^q$ and suppose also that $p$ is a positive integer.
Let $\Lambda\subseteq\real^{d-1}$
consisting of $(d-1)$-tuples $\bar l = (l_1, \ldots, l_{d-1})$
satisfying the following conditions: (i) For $k$ satisfying $1\le k\le d-1$,
$0\le l_k \le p$;
(ii) For $k$ satisfying $1\le k \le \frac d2 -1$, $l_k - l_{2k} - l_{2k+1}$
is a nonnegative, even integer.
Let $S\subseteq \real^{d-1}$ denote the convex polytope defined by the
following inequalities:
\begin{eqnarray*}
&&l_1 \le p,\\
&&l_k\ge l_{2k}+l_{2k+1} \qquad (1\le k < \tfrac d2),\\
&&l_k\ge 0 \qquad (\tfrac d2 \le k < d).
\end{eqnarray*}
If $L$ denotes the lattice $L \subseteq \intg^{d-1}$
consisting of $(d-1)$-tuples of integers for which $l_k - l_{2k} - l_{2k+1}$ is even for $1\le k < \frac d2$, then
$\Lambda = L \cap S$.
(The polytope $S$ is a $(d-1)$-simplex, as shown below.)
For $1\le k < \frac d2$, put $n_k = \frac12 (l_k-l_{2k}-l_{2k+1})$;
for $\frac d2 \le k < d$, put $n_k = l_k$.
The function $\varphi:\real^d\rightarrow\real^d$ taking $\bar l$ to $\bar n = (n_1,\ldots, n_{d-1})$ is
an invertible linear transformation.
In particular, $l_1$ is given
in terms of the $n_k$'s by
\begin{eqnarray*}
&&l_1 = 2(n_1+\cdots + n_{\frac d2 - 1}) + n_{\frac d2} + \cdots + n_{d-1}.
\end{eqnarray*}
The transformation takes $S$ to the polytope $T = \varphi(S)$
defined by
\begin{eqnarray*}
&&n_k \ge 0  \qquad(1 \le k < d),\\
&&2(n_1+\cdots + n_{\frac d2 - 1}) + n_{\frac d2} + \cdots + n_{d-1}\le p.
\end{eqnarray*}
The polytope $T$ is a simplex, being the subset of the first orthant
in $\real^{d-1}$ that is cut off by a hyperplane.  Therefore the linearly equivalent set $S$ is
also a simplex.
The function $\varphi$ takes $L$ to $\intg^{d-1}$ and consequently it takes $\Lambda := L\cap S$
to $\Gamma := \intg^{d-1} \cap T$.
For a set $\{x_{\bar n}\}$ indexed by $\Gamma$, one has
 $\sum_{\bar n\in \Gamma} x_{\bar n}
= \sum_{\bar l\in \Lambda} x_{\varphi(\bar l)}$.
This completes the proof.
\end{proof}
}

\begin{lem} \label{sumHopfinfty}
Let $p\in\N_0$.
$1 \le j \le \fdt$,
$1 \le k \le \fdt-1$.
Consider the multi-sum ${\sf Z}_2$ defined by
\begin{eqnarray}
&&{\sf Z}_2:=\sum_{l=p+1}^\infty 
\sum_{l_2=0}^\infty 
\cdots
\sum_{l_{2^{q-1}-1}=0}^\infty
\ 
\sum_{m_{2^{q-1}}=0}^\infty
\cdots
\sum_{m_2=0}^\infty
\ 
\sum_{m=0}^p
\end{eqnarray}
over $\mcL\cup\mcM$ with 
$0\le m_j\le p$,
$l_k\ge p+1$,
with the same restriction over the $l_k\in\mcL$ in the above lemma.
Then ${\sf Z}_2$ can be re-expressed as a multi-sum over $\mcN\cup\mcM$ with sum order reversed to obtain
\begin{eqnarray}
&&\hspace{0cm}{\sf Z}_2=\sum_{m=0}^p
\sum_{m_2=0}^\infty
\cdots
\sum_{m_{2^{q-1}}=0}^\infty
\ 
\sum_{n_{2^{q-1}-1}=0}^\infty
\cdots
\sum_{n_2=0}^\infty
\ 
\sum_{n=\max\left(0,\left\lfloor \frac{p-{\scriptstyle \sum_{i=1}^{2^{q-1}}m_i}-2\sum_{i=2}^{2^{q-1}-1}n_i}{2}\right\rfloor+1\right)}^\infty.
\end{eqnarray}
\end{lem}
{
\begin{proof}
Let $S^\prime\subseteq \real^{d-1}$ denote the convex polyhedron defined by the
following inequalities.
\begin{eqnarray*}
&&l_1\ge p+1,\\
&&l_k\ge l_{2k}+l_{2k+1} \qquad (1\le k < \tfrac d2),\\
&&p\ge l_k\ge 0 \qquad(\tfrac d2 \le k < d).
\end{eqnarray*}
The polyhedron $T^\prime = \varphi(S^\prime)$ is given by the following inequalities.
\begin{eqnarray*}
&&n_k \ge 0 \qquad (1\le k < d),\\
&&n_k \le p \qquad (\frac d2 \le k < d),\\
&&2(n_1+\cdots + n_{\frac d2 - 1}) + n_{\frac d2} + \cdots + n_{d-1}
\ge p+1.
\end{eqnarray*}
The set $\Lambda^\prime := S^\prime \cap L$ is mapped to $\Gamma^\prime := T^\prime \cap \intg^{d-1}$ by $\varphi$.
For a set $\{x_{\bar n}\}$ indexed by $\Gamma^\prime$, one has
 $\sum_{\bar n\in \Gamma^\prime} x_{\bar n}
= \sum_{\bar l\in \Lambda^\prime} x_{\varphi(\bar l)}$.
This completes the proof.
\end{proof}
}

\begin{lem} \label{multisumZ3}
Let $m,p,k\in\N_0$ such that
$p=k-d/2\in\N_0$, 
$d\in 2\N$ and
$1 \le k \le \fdt-1$.
Consider the following multi-sum over $\mcL\cup\mcM$ with
$m\ge p+1$.
\begin{eqnarray}
&&{\sf Z}_3:=
\sum_{l=p+1}^\infty 
\sum_{l_2=0}^\infty
\cdots
\sum_{l_{\fdt-1}=0}^\infty
\sum_{m_\fdt=0}^\infty
\cdots
\sum_{m_2=0}^\infty
\sum_{m=p+1}^\infty,
\end{eqnarray}
with the same restriction over the $l_k\in\mcL$ in the above lemmas.
Then the  multi-sum ${\sf Z}_3$ can be re-expressed as a multi-sum over $\mcN\cup\mcM$ with sum order reversed to obtain
\begin{eqnarray}
&&{\sf Z}_3
%
=
\sum_{m=p+1}^\infty
\sum_{m_2=0}^\infty
\cdots
\sum_{m_\fdt=0}^\infty
\sum_{n_{\fdt-1}=0}^\infty
\cdots
\sum_{n=0}^\infty.
\end{eqnarray}
\end{lem}
{
\begin{proof}
Let $S^{\prime\prime}\subseteq \real^{d-1}$ denote the convex polyhedron defined by the
following inequalities:
\begin{eqnarray*}
&&l_k\ge l_{2k}+l_{2k+1}, \qquad (1\le k < \tfrac d2),\\
&&l_k\ge 0, \qquad (\tfrac d2 \le k < d),\\
&&l_{d-1}\ge p+1.
\end{eqnarray*}
The polyhedron $T^{\prime\prime} = \varphi(S^{\prime\prime})$ is given by the following inequalities:
\begin{eqnarray*}
&&n_k \ge 0 \qquad (1 \le k < d),\\
&&n_{d-1} \ge p+1.
\end{eqnarray*}
The set $\Lambda^{\prime\prime} := S^{\prime\prime} \cap L$ is mapped to $\Gamma^{\prime\prime} := T^{\prime\prime} \cap \intg^{d-1}$ by $\varphi$.
For a set $\{x_{\bar n}\}$ indexed by $\Gamma^{\prime\prime}$, one has
\[
\sum_{\bar n\in \Gamma^{\prime\prime}} x_{\bar n}
= \sum_{\bar l\in \Lambda^{\prime\prime}} x_{\varphi(\bar l)}.
\]
This completes the proof.
\end{proof}
}

\section{
Binomial and logarithmic addition theorems
for the azimuthal Fourier coefficients}
\label{AdditionthmsinVilenkinspolysphericalcoordinates} 

In even-dimensional Euclidean space $\R^d$, the kernels for a fundamental solution of the
polyharmonic equation are represented by a function of the distance between two 
points \eqref{greenpoly}, $\|\bfx-\bfxp\|$. In rotationally-invariant coordinates systems, one 
may represent the distance between the source and observation points in terms of the difference 
between azimuthal coordinates, namely
\eqref{rotationallyinvariantdist}. In Vilenkin's polyspherical coordinates, which are 
rotationally-invariant, one may also represent this distance in terms of the separation 
angle, namely \eqref{sphericallysymmetricdist}.
In a previous publication \cite{Cohl12pow}, addition 
theorems which arise when a fundamental solution of the polyharmonic equation is of a power-law in Euclidean space with odd-dimensions and 
for even-dimensions when $1\le k<d/2$ were derived. In this manuscript, we treat the even-dimensional
case for $k\ge d/2$, in which the functional dependence is either logarithmic or of binomial form.
By considering the equivalent azimuthal $\phi-\phi'$ and 
separation angle $\gamma$, Fourier and Gegenbauer expansions 
respectively of the kernels for the polyharmonic equation in 
even-dimensional space, can derive addition theorems for the 
azimuthal Fourier coefficients. 

The procedure for developing the addition theorems in this section is as follows.
Let $p=k-d/2\in\N_0,$ with $f_p:\R^d\times\R^d\to\R$,
$g_p:(\R^d\times\R^d)\setminus\{(\bfx,\bfx):\bfx\in\R^d\}\to\R$,
defined such that $f_p(\bfx,\bfxp):=\|\bfx-\bfxp\|^{2p}$, $g_p(\bfx,\bfxp):=\|\bfx-\bfxp\|^{2p}\log\|\bfx-\bfxp\|$. 
First, express $f_p,g_p$ in terms of their azimuthal separation angle Fourier cosine series
using \eqref{Eulerpolyexp}, \eqref{limitderivCheby}.  Call these the left-hand sides. Then 
express $f_p,g_p$ in terms of their separation angle Gegenbauer expansions 
given using 
\eqref{binomGeg}, \eqref{limitderivGeg} 
with $\mu=d/2-1$.
Call these the right-hand sides. By using the addition theorem for hyperspherical harmonics
\eqref{addthmhypsph}, we can expand the right-hand sides in terms of a product of 
separable harmonics in a chosen Vilenkin polyspherical coordinate system. 
Since Vilenkin's polyspherical
coordinate systems are rotationally-invariant, one of the coordinates will
correspond to the chosen azimuthal separation angle which
has been expanded about on the left-hand side of the azimuthal 
Fourier expansion. 
To obtain the addition theorem, one must re-arrange the 
multi-sum expression which arises on the right-hand side so 
that the outermost sum is the sum over the relevant azimuthal 
quantum number. Addition theorems are simply derived by comparing 
the azimuthal Fourier coefficients on both sides.

In order to obtain binomial and logarithmic addition theorems in a Vilenkin polyspherical
coordinate system, we relate respectively $\|\bfx-\bfxp\|^{2p}$ and
$\|\bfx-\bfxp\|^{2p}\log\|\bfx-\bfxp|$, in terms of their Fourier 
cosine series over the azimuthal separation angle, and their 
Gegenbauer polynomial expansions over the separation 
angle \eqref{sepang}. These equalities reduce respectively to
\begin{equation}
\left(\chi-\cos(\phi-\phi')\right)^p=\left(\frac{rr'}{RR'}\right)^p\left(\zeta-\cos\gamma\right)^p,
\label{binomaddthmform}
\end{equation}
and
\begin{eqnarray}
&&\hspace{-0.6cm}\log(2RR')\left(\chi-\cos(\phi-\phi')\right)^p+(\chi-\cos(\phi-\phi'))^p\log(\chi-\cos(\phi-\phi'))\nonumber\\[0.2cm]
&&\hspace{2cm}=\left(\frac{rr'}{RR'}\right)^p\left(
\log(2rr')\left(\zeta-\cos\gamma\right)^p
+(\zeta-\cos\gamma)^p\log(z-\cos\gamma) \right).
\label{logddthmform}
\end{eqnarray}

\subsection{Addition theorems in standard polyspherical coordinates}
{The binomial addition theorems in standard polyspherical coordinates are given by the following two theorems.
\begin{thm} \label{binomial1}
{Let $m\in\N_0$, $p=k-d/2\in\N_0$, $d\in2\N$, $0 \le m \le p$.}  Then
\begin{align*}
Q_{m-\frac12}^{p+\frac12}(\chi) &=
(-1)^{\frac{d}{2}-1+m}(2\pi)^{\frac{d}{2}-1}  (\chi^2-1)^{-\frac{p}{2}-\frac14}(p-m)!(p+m)!\left( \frac{ r r'}{R R'} \right)^p \pdh \nonumber \\
& \times \sum_{l_{d-2}=m}^p      \OmeH{d-2}{m}\cdots \sum_{l=l_2}^{p} \Omeh{1}{l_2}
\frac{(-1)^l}
{(p-l)!(l+p+d-2)!} Q_{l+\frac{d-3}{2}}^{p+\frac{d-1}{2}}(\zeta)
\end{align*}
\end{thm}
\begin{proof}
This equality can be found by comparing the binomial expansions of Theorems \ref{bifourier} and \ref{bigeg}.  The Gegenbauer polynomials in Theorem \ref{bigeg} can be expanded with the aid of \eqref{addthmhypsph}.  The normalized hyperspherical harmonics can be written in standard polyspherical coordinates, as seen in \eqref{normhyper}.  Using the concise notation of \eqref{standpoly}, the Gegenbauer expansion of \eqref{addthmhypsph} can be written as follows:
\begin{equation} \label{gegasomega}
    C_l^{d/2-1}(\cos \gamma) = \frac{(d-2)\pi^{\frac{d}{2}-1}}{(2l+d-2)\Gamma(d/2)} \sum_{K} \expe^{im(\phi-\phi')} \OmeH{d-2}{m}\cdots \Omeh{1}{l_2}.
\end{equation}
Now the expansion of Theorem \ref{bigeg} will contain a multi-sum.  Our goal is to compare the Fourier coefficients of the binomial expansions, this multi-sum will need to be reversed. This reversal is shown in Lemma \ref{multisum}.
This allows us to compare the Fourier coefficients, and complete the proof.
\end{proof}
\begin{thm} \label{binomial2}
Let $m\in\N_0$, $k\in\N$, $p=k-d/2\in\N_0$,
$d=4$ and $0\le m \le p$.  Then
\begin{align*}
Q_{m-\frac12}^{p+\frac12}(\chi) &= (-1)^{1-m} (\chi^2-1)^{-\frac{p}{2}-\frac14} (p-m)!(p+m)! \left( \frac{ r r'}{R R'} \right)^p \pth  \nonumber \\
& \times \sum_{l_2=m}^p   \frac{2^{2l_2} \left( l_2! \right)^2 (2l_2+1)(l_2-m)!}{  (l_2+m)!} 
\left( \sin{\theta} \sin{\theta'} \right)^{l_2}
{\sf P}_{l_2}^m(\cos{\theta_2}){\sf P}_{l_2}^m(\cos{\theta_2'}) \nonumber \\
& \times    \sum_{l=l_2}^{p} 
\frac{(-1)^l(2l+2)(l-l_2)!}{(p-l)!(l+p+2)! (l+l_2+1)!} 
C_{l-l_2}^{l_2+1}(\cos{\theta}) C_{l-l_2}^{l_2+1}(\cos{\theta'})
Q_{l+\frac{1}{2}}^{p+\frac{3}{2}}(\zeta).
\end{align*}
\end{thm}
\begin{proof}
To prove this, take the result of Theorem \ref{binomial1} and use $d=4$.  This simplifies the multi-sum of $\Omega$ functions:
\begin{align} \label{d4omega}
    \sum_{l_{d-2}}\OmeH{d-2}{m}\cdots
\sum_{l}\Omega_{1}^{d}
\left(
l,l_2
;\!\begin{array}{c} {\theta} \\[0pt] {\theta'}
\end{array}
\right) \rightarrow \sum_{l_2} 
\Omega_{2}^{4}
\left(
l_2,m
;\!\begin{array}{c} {\theta_2} \\[0pt] {\theta'_2}
\end{array}
\right)
\sum_l
\Omega_{1}^{4}
\left(
l,l_2
;\!\begin{array}{c} {\theta} \\[0pt] {\theta'}
\end{array}
\right).
\end{align}
By \eqref{standpoly}, we know the $\Omega$-functions are products of $\Theta$-functions, which are defined in \eqref{thetaGeg}. A couple of the Gegenbauer polynomials that appear can be rewritten as Ferrers functions using \eqref{JacobiGeg}.  This allows \eqref{d4omega} to be written as:
\begin{eqnarray} \label{d4omegaF}
&& \hspace{-1cm}  \sum_{l_2} 
\Omega_{2}^{4}
\left(
l_2,m
;\!\begin{array}{c} {\theta_2} \\[0pt] {\theta'_2}
\end{array}
\right)
\sum_l
\Omega_{1}^{4}
\left(
l,l_2
;\!\begin{array}{c} {\theta} \\[0pt] {\theta'}
\end{array}
\right) = \frac{1}{\pi} \sum_{l_2} \frac{  2^{2l_2}\left( l_2! \right)^2 (2l_2+1)(l_2-m)!}{ (l_2+m)!} \nonumber\\
&&\hspace{1cm}\times (\sin \theta \sin \theta')^{l_2} {\sf P}_{l_2}^m ( \cos \theta_2) {\sf P}_{l_2}^m ( \cos \theta'_2)\sum_l
 \frac{(l+1)(l-l_2)!}{(l+l_2+1)!} C_{l-l_2}^{l_2+1}(\cos \theta)  C_{l-l_2}^{l_2+1}(\cos \theta').
\end{eqnarray}
Then simplification completes the proof.
\end{proof}
}
In standard polyspherical coordinates, one has the following {logarithmic}
addition theorem {for $0\le m\le p$}.
\begin{thm} \label{LogPSexpand}
{Let $m\in\N_0$, $k\in\N$, $p=k-d/2\in\N_0$,
$d\in2\N$ and $0\le m \le p$.  Then}
{
\begin{eqnarray}
&& \hspace{-0.6cm}\frac{1}{(p+m)!(p-m)!}
\left(\log(RR')
+\log(\chi+\sqrt{\chi^2-1})+2H_{2p}-H_{p+m}-H_{p-m}
\right)
Q_{m-\frac12}^{p+\frac12}(\chi)\nonumber\\
&&\hspace{-0.4cm} +\frac{1}{(p+m)!}
\sum_{k=0}^{p-m-1}
\frac{(2m+2k+1)}
{k!(p-m-k)(p+m+k+1)}
\left[1+\frac{k!(p+m)!}{(2m+k)!(p-m)!}\right]
Q_{m-\frac12}^{m+k+\frac12}(\chi)
\nonumber\\
&&\hspace{-0.4cm} +\frac{1}{(p-m)!}
\sum_{k=0}^{m-1}
\frac{2k+1}{(m+k)!(p-k)(p+k+1)}Q_{m-\frac12}^{k+\frac12}(\chi)\nonumber\\
&&\hspace{-0.4cm}=\left(\frac{rr'}{RR'}\right)^p(-1)^{m+\fdt-1}(2\pi)^{\fdt-1}
\frac{(\zeta^2-1)^{\frac{p}2+\frac{d-1}{4}}}
{(\chi^2-1)^{\frac{p}{2}+\frac14}}\nonumber\\
&&\hspace{-0.1cm}\times\biggl\{
\sum_{l_{d-2}=m}^p\OmeH{d-2}{m}\cdots
\sum_{l=l_2}^p\Omeh{1}{l_2}\nonumber\\
&&\hspace{0.4cm}\times\biggl[ 
\frac{(-1)^lQ_{l+\frac{d-3}{2}}^{p+\frac{d-1}{2}}(\zeta)}{(p-l)!(p+l+d-2)!}
\left(
\log(rr')+\log(\zeta\!+\!\sqrt{\zeta^2-1})\!+\!2H_{2p+d-2}\!+\!H_p-H_{p+\fdt-1}
\!-\!H_{p+l+d-2}\!-\!H_{p-l}
\right)\nonumber\\
&&\hspace{0.9cm}+\frac{(-1)^l}{(p+l+d-2)!}\sum_{k=0}^{p-l-1}
\frac{(2l+2k+d-1)}
{k!(p-l-k)(p+l+k+d-1)}
\left[1+\frac{k!(p+l+d-2)!}{(k+2l+d-2)!(p-l)!}\right]
Q_{l+\frac{d-3}{2}}^{k+l+\frac{d-1}{2}}(\zeta)
\nonumber\\
&&\hspace{0.9cm}+\frac{(-1)^l}{(p-l)!}\sum_{k=0}^{l+\fdt-2}
\frac{2k+1}{(l+\fdt+k-1)!(p+\fdt-k-1)(p+\fdt+k)}
Q_{l+\frac{d-3}{2}}^{k+\frac12}(\zeta)\biggr]\nonumber\\
&&\hspace{0.4cm}+\sum_{l_{d-2}=m}^\infty \OmeH{d-2}{m}\cdots
\sum_{l_2=l_3}^\infty \OmeH{2}{l_3}\nonumber\\
&&\hspace{1.1cm}\times\sum_{l=\text{max}(p+1,l_2)}^\infty\frac{(-1)^{p+1}(l-p-1)!}
{(p+l+d-2)!}
\Omeh{1}{l_2} Q_{l+\frac{d-3}{2}}^{p+\frac{d-1}{2}}(\zeta)\biggr\}.
\label{thm12it}
\end{eqnarray}}
\end{thm}
\begin{proof}
{
To prove this theorem, we take the Fourier and Gegenbauer expansions of logarithmic fundamental solutions, seen in Theorems \ref{fourierlog} and \ref{gegexpand} respectively.  For the Gegenbauer expansion, the substitution of $\mu = d/2 -1$ is used.
Again, the Gegenbauer polynomials can be rewritten as in \eqref{gegasomega}.
The Gegenbauer expansion contains three different multi-sums, one contains the sum of $l$ from zero to $p$, the next contains the sum of $l$ from zero to $p+1$, and the last contains the sum of $l$ from $p+1$ to infinity.  The first two of these sums can be reversed as in Lemma \eqref{multisum}, where the later $p$ can be replaced with $p+1$.
For the final multi-sum, {Lemma \ref{lp1sumlem} is used. Since only the region $0\le m \le p$ is required, only the first term of the split multi-sum is used.}
When comparing the Fourier coefficients of the expansions, only those terms where $0 \le m \le p$ are used.  Thus only the first two parts of the above multi-sum are used.
}
\end{proof}

The following corollary results substituting $d=4$ in
the above logarithmic addition theorem \eqref{thm12it}.

\begin{cor}
Let $0\le m\le p$, $p=k-d/2\in\N_0$, $k\in\N$ and $d=4$. Then
\begin{eqnarray}
&& \hspace{-0.5cm}\frac{1}{(p+m)!(p-m)!}
\left(\log(RR')
+\log(\chi+\sqrt{\chi^2-1})+2H_{2p}-H_{p+m}-H_{p-m}
\right)
Q_{m-\frac12}^{p+\frac12}(\chi)\nonumber\\
&& \hspace{-0.25cm}+\frac{1}{(p+m)!}
\sum_{k=0}^{p-m-1}
\frac{(2m+2k+1)
}
{k!(p-m-k)(p+m+k+1)}
\left[1+\frac{k!(p+m)!}{(2m+k)!(p-m)!}\right]
Q_{m-\frac12}^{m+k+\frac12}(\chi)
\nonumber\\
&& \hspace{-0.25cm}+\frac{1}{(p-m)!}
\sum_{k=0}^{m-1}
\frac{2k+1}{(m+k)!(p-k)(p+k+1)}Q_{m-\frac12}^{k+\frac12}(\chi)\nonumber\\
&&=2\,(-1)^{m+1}
\left(\frac{rr'}{RR'}\right)^p
\frac{(\zeta^2-1)^{\frac{p}2+\frac{3}{4}}}
{(\chi^2-1)^{\frac{p}{2}+\frac14}}
\biggl\{
\sum_{l_{2}=m}^p \frac{(2l_2+1)(l_2-m)!}{(l_2+m)!}
{\sf P}_{l_2}^m(\cos\theta_2){\sf P}_{l_2}^m(\cos\theta_2')
2^{2l_2}(l_2!)^2(\sin\theta\sin\theta')^{l_2}\nonumber\\
&&\hspace{0.5cm}\times\sum_{l=l_2}^p \frac{(-1)^l(l+1)(l-l_2)!}{(l+l_2+1)!}
C_{l-l_2}^{l_2+1}(\cos\theta)
C_{l-l_2}^{l_2+1}(\cos\theta')\nonumber\\
&&\hspace{1.0cm}\times\biggl[ 
\frac{1}{(p-l)!(p+l+2)!}\left(\log(rr')+\log(\zeta+\sqrt{\zeta^2-1})
+2H_{2p+2}+H_p-H_{p+1}-H_{p+l+2}-H_{p-l}\right)Q_{l+\fdt}^{p+\frac32}(\zeta)\nonumber\\
&&\hspace{1.5cm}+\frac{1}{(p+l+2)!}\sum_{k=0}^{p-l-1}
\frac{(2l+2k+3)
}
{k!(p-l-k)(p+l+k+3)}
\left[1+\frac{k!(p+l+2)!}{(k+2l+2)!(p-l)!}\right]
Q_{l+\frac{1}{2}}^{k+l+\frac{3}{2}}(\zeta)
\nonumber\\
&&\hspace{1.5cm}+\frac{1}{(p-l)!}\sum_{k=0}^{l}
\frac{2k+1}{(l+k+1)!(p-k+1)(p+k+2)}
Q_{l+\frac{1}{2}}^{k+\frac12}(\zeta)\biggr]\nonumber\\
&&\hspace{0.5cm}+(-1)^{p+1}\sum_{l_{2}=m}^\infty 
\frac{(2l_2+1)(l_2-m)!}{(l_2+m)!}
{\sf P}_{l_2}^m(\cos\theta_2){\sf P}_{l_2}^m(\cos\theta_2')
2^{2l_2}(l_2!)^2(\sin\theta\sin\theta')^{l_2}\nonumber\\
&&\hspace{1.0cm}\times\sum_{l=\max(l_2,p+1)}^\infty \frac{(l+1)(l-l_2)!}{(l+l_2+1)!}
C_{l-l_2}^{l_2+1}(\cos\theta)
C_{l-l_2}^{l_2+1}(\cos\theta')
\frac{(l-p-1)!}{(p+l+2)!}
Q_{l+\frac{1}{2}}^{p+\frac{3}{2}}(\zeta)\biggr\}.
\label{addnthm7.4}
\end{eqnarray}
\end{cor}
\begin{proof}
{ 
To prove this, take the result of Theorem \ref{LogPSexpand} for the $d=4$ case.  The $\Omega$-functions can be simplified as in the proof of Theorem \ref{binomial2}, as seen in \eqref{d4omegaF}.
This completes the proof.}
 \end{proof}

\noindent In standard polyspherical coordinates, one has the following {logarithmic}
addition theorem {for $m\ge p+1$}.

 {
\begin{thm} \label{sphereExpandmgp}
Let $m\in\N_0$, $p=k-d/2\in\N_0$, $d\in2\N$, $m \ge p+1$.  Then
\begin{align}
    Q_{m-\frac12}^{p+\frac12}(\chi) &= (-1)^{\frac{d}{2}-1} (2 \pi)^{\frac{d}{2}-1} (\zeta^2-1)^{\frac{p}{2}+\frac{d-1}{4}}(\chi^2-1)^{-\frac{p}{2}-\frac14}\frac{(p+m)!}{(m-p-1)!} \left( \frac{ r r'}{R R'}\right)^p \nonumber \\
    & \times \sum_{l_{d-2}=p+1}^{\infty} \OmeH{d-2}{m}\cdots \sum_{l=l_2}^{\infty} \Omeh{1}{l_2} \frac{(l-p-1)!}
{(p+l+d-2)!}
Q_{l+\frac{d-3}{2}}^{p+\frac{d-1}{2}}(\zeta).
\label{addt7.5}
\end{align}
\end{thm}
\begin{proof}
Again, this proofs starts with the equal logarithmic expansions of \ref{fourierlog} and \ref{gegexpand}.  With the Gegenbauer expansion, the polynomials are rewritten as seen in \eqref{gegasomega}.  The Fourier coefficients of interest only appear in the multi-sum that has $l$ summing over $p+1$ to infinity.  
This multi-sum can be seen in Lemma \ref{lp1sumlem}.  The second term of this multi-sum contains values of $m$ that are greater than $p$, which are the terms needed for the comparison to complete the proof.
\end{proof}
\begin{cor}
Let $m\in\N_0$, $p=k-d/2\in\N_0$, $d=4$, $m \ge p+1$.  Then
\begin{eqnarray}
&&\hspace{-1.25cm}   Q_{m-\frac12}^{p+\frac12}(\chi) = -2 (\zeta^2-1)^{\frac{p}{2}+\frac{3}{4}}(\chi^2-1)^{-\frac{p}{2}-\frac14}\frac{(p+m)!}{(m-p-1)!} \left( \frac{ r r'}{R R'}\right)^p \nonumber \\
    && \hspace{0.5cm}\times \sum_{l_2=p+1}^{\infty} \frac{  2^{2l_2}\left( l_2! \right)^2 (2l_2+1)(l_2-m)!}{ (l_2+m)!} \nonumber\\
&&\hspace{0.75cm}\times (\sin \theta \sin \theta')^{l_2} {\sf P}_{l_2}^m ( \cos \theta_2) {\sf P}_{l_2}^m ( \cos \theta'_2)\sum_{l=l_2}^{\infty}
 \frac{(l+1)(l-l_2)!}{(l+l_2+1)!} C_{l-l_2}^{l_2+1}(\cos \theta)  C_{l-l_2}^{l_2+1}(\cos \theta')
Q_{l-\frac12}^{p+\frac32}(\zeta).
\end{eqnarray}
\end{cor}
\begin{proof}
Start with 
\eqref{addt7.5} and set $d=4$.  The $\Omega$-functions can be simplified as in \eqref{d4omegaF}, which completes 
the proof.
\end{proof}
}
\subsection{Addition theorems in generalized Hopf coordinates}


In generalized Hopf coordinates, one has the following 
{binomial} addition theorem valid for $0\le m \le p$.

\begin{thm}
\label{baHopf}
Let $p,m\in\N_0$, $0\le m\le p$, $d=2^q$, $q\ge 2$.
Then
\begin{eqnarray}
&&\hspace{-1cm}Q_{m-\frac12}^{p+\frac12}(\chi)=(-1)^{m+\fdt-1}2^{\fdt-1}(p-m)!(p+m)!(\chi^2-1)^{-\frac{p}{2}-\frac14}
\rrRRp \pdh\nonumber\\[0.2cm]
&&\hspace{-0.5cm}\times \sum_{m_2=0}^{p-m}\epsilon_{m_2}\cos(m_2(\phi_2-\phi_2'))
\cdots
\sum_{m_k=0}^{p-\sum_{i=1}^{k-1}m_i}
\cdots 
\sum_{m_\fdt=0}^{p-\sum_{i=1}^{\fdt-1}m_i} \epsilon_{m_\fdt}\cos(m_\fdt(\phi_\fdt-\phi_\fdt'))\nonumber\\[0.2cm]
&&\hspace{-0.5cm}\times(-1)^M\sum_{n_{\fdt-1}=0}^{\left\lfloor\frac{p-M}{2}\right\rfloor} \PsiH{\fdt-1}{\log_2d}{m_2,m}
\cdots
\sum_{n=0}^{\left\lfloor\frac{p-M-2\sum_{i=2}^{\fdt-1}n_i}{2}\right\rfloor} \PsIH{1}{\log_2d}{l_2,l_3}\nonumber\\[0.2cm]
&&\hspace{-0.5cm}\times \frac{1}{(p-2N-M)!(p+2N+M+d-2)!}
Q_{2N+M+\frac{d-3}{2}}^{p+\frac{d-1}{2}}(\zeta).
\end{eqnarray}
\end{thm}
\begin{proof}
{Start with the Gegenbauer expansion of a fundamental solution given by Theorem \ref{bigeg}. The hyperspherical harmonics contained within can be expanded with the aid of \eqref{addthmhypsph} and Theorem \ref{hopfaddthm} to be
\begin{eqnarray} \label{batgegharm}
&& \hspace{-0.5cm} C_l^{d/2-1}(\cos\gamma) = \frac{(d-2)}
{(2l+d-2)\Gamma(d/2)} \sum_{K} 
\epsilon_m\cos(m(\phi-\phi'))
\epsilon_{m_2}\cos(m_2(\phi_{m_2}-\phi_{m_2}'))\cdots
\epsilon_{m_\fdt}\cos(m_\fdt(\phi_\fdt-\phi_\fdt'))\nonumber\\
&&
\hspace{3cm}\times\PsiH{\fdt-1}{\log_2d}{m_2,m}
\times\cdots\times
\PsIH{1}{\log_2d}{l_2,l_3}.
\end{eqnarray}
The multi-sum in  \eqref{batgegharm} can be combined with the sum in \eqref{bigeg}, and with the ability to reverse this multi-sum given by Lemma \ref{sumreversal} yields: 
\begin{eqnarray}
&& \hspace{-0.5cm}\hii_k^d(\bfx,\bfxp)= \frac{2^{\frac{d}{2}-\frac12}}{\sqrt{\pi}}\expe^{i\pi(p-\frac{d}{2}+\frac12)}p! (2rr)^p \pdh  \sum_{m=0}^{p-m}\epsilon_{m}\cos(m(\phi-\phi')) \nonumber \\
&&\hspace{0.5cm}\times  \sum_{m_2=0}^{p-m}\epsilon_{m_2}\cos(m_2(\phi_2-\phi_2'))
\cdots
\sum_{m_k=0}^{p-\sum_{i=1}^{k-1}m_i}
\cdots 
\sum_{m_\fdt=0}^{p-\sum_{i=1}^{\fdt-1}m_i} \epsilon_{m_\fdt}\cos(m_\fdt(\phi_\fdt-\phi_\fdt'))\nonumber\\
&&\hspace{1.5cm}\times(-1)^M\sum_{n_{\fdt-1}=0}^{\left\lfloor\frac{p-M}{2}\right\rfloor} \PsiH{\fdt-1}{\log_2d}{m_2,m}
\cdots
\sum_{n=0}^{\left\lfloor\frac{p-M-2\sum_{i=2}^{\fdt-1}n_i}{2}\right\rfloor} \PsIH{1}{\log_2d}{l_2,l_3}\nonumber\\
&& \hspace{1.5cm} \times \frac{1}{(p-2N-M)!(2N + M + p + d -2)!}Q_{l+\frac{d-3}{2}}^{p+\frac{d-1}{2}}(\zeta).
\end{eqnarray}
The Fourier coefficients of this can be compared with the Fourier expansion of a fundamental solution given in Theorem \ref{bifourier}.  This completes the proof
}
\end{proof}


\noindent The simplest example of a {binomial addition theorem for $0 \le m\le p$ in generalized Hopf coordinates occurs in four dimensions}.

\begin{thm} \label{biHopf}
Let $p,m\in\N_0$, $0\le m\le p$, $p=k-d/2\in\N_0$ and $d=4$.
Then
\begin{eqnarray}
&&\hspace{-1cm}Q_{m-\frac12}^{p+\frac12}(\chi)=-2(p-m)!(p+m)!(\chi^2-1)^{-\frac{p}{2}-\frac14}
\rrRRp \pth (\sin\vartheta\sin\vartheta')^{m}
\nonumber\\[0.2cm]
&&\hspace{-0.5cm}\times \sum_{m_2=0}^{p-m}\epsilon_{m_2}\cos(m_2(\phi_2-\phi_2'))
(-1)^{m_2}
(\cos\vartheta\cos\vartheta')^{m_2}
\sum_{n=0}^{\left\lfloor\frac{p-m-m_2}{2}\right\rfloor}
P_n^{(m,m_2)}(\cos(2\vartheta))
P_n^{(m,m_2)}(\cos(2\vartheta'))
\nonumber\\[0.2cm]
&&\hspace{-0.5cm}\times
\frac{(2n+m+m_2+1)(n+m+m_2)!n!
}
{(p-2n-m-m_2)!(p+2n+m+m_2+2)!(n+m)!(n+m_2)!}
Q_{2n+m+m_2+\frac12}^{p+\frac32}(\zeta).
\end{eqnarray}
\end{thm}
\begin{proof}
{Starting with Theorem \ref{baHopf}, let $d=4$.  From this we obtain the following:
\begin{eqnarray} \label{batd4}
&&\hspace{-1cm}Q_{m-\frac12}^{p+\frac12}(\chi)=-2(p-m)!(p+m)!(\chi^2-1)^{-\frac{p}{2}-\frac14}
\rrRRp \pth\nonumber\\[0.2cm]
&&\hspace{-0.5cm}\times \sum_{m_2=0}^{p-m}(-1)^{m_2}\epsilon_{m_2}\cos(m_2(\phi_2-\phi_2')) \sum_{n=0}^{\left\lfloor\frac{p-m-m_2}{2}\right\rfloor}  \Psi_{1}^{2}
\left(
\begin{array}{c} {n} \\ {m,m_2}
\end{array};
\begin{array}{c} {\vartheta} \\[5pt] {\vartheta'}
\end{array}
\right)\nonumber\\[0.2cm]
&&\hspace{0.5cm}\times \frac{1}{(p-2n-m-m_2)!(p+2n+m+m_2+2)!}
Q_{2n+m+m_2+\frac12}^{p+\frac32}(\zeta).
\end{eqnarray}
Then by \eqref{UpsDef}, \eqref{psi2ups}, the function denoted by $\Psi$ can be written in terms of Jacobi polynomials: 
\begin{eqnarray} \label{PsiHopf}
&&\hspace{-1cm} \Psi_{1}^{2}
\left(
\begin{array}{c} {n} \\ {m,m_2}
\end{array};
\begin{array}{c} {\vartheta} \\[5pt] {\vartheta'}
\end{array}
\right)   = \frac{(2n + m + m_2+1)(n+m+m_2)!n!}{(n+m)!(n+m_2)!} ( \sin\vartheta \sin\vartheta' )^m (\cos \vartheta \cos \vartheta')^{m_2}\nonumber \\[0.2cm]
&&\hspace{3.5cm}\times P_n^{(m,m_2)}(\cos(2\vartheta)) P_n^{(m,m_2)}(\cos(2 \vartheta')).
\end{eqnarray}
Using this identity with \eqref{batd4} completes the proof.
}
\end{proof}



In generalized Hopf coordinates, one has the following 
{logarithmic} addition theorem valid for $0\le m \le p$.

\begin{thm} \label{logHopf}
Let $0\le m\le p$, $d=2^q$, $q\ge 2$ and  $p=k-d/2\in\N_0$. Then
\begin{eqnarray}
&&\hspace{-1.0cm} \frac{1}{(p+m)!(p-m)!}
\left(\log(RR')
+\log(\chi+\sqrt{\chi^2-1})+2H_{2p}-H_{p+m}-H_{p-m}
\right)
Q_{m-\frac12}^{p+\frac12}(\chi)\nonumber\\
&&\hspace{-1.0cm} +\frac{1}{(p+m)!}
\sum_{k=0}^{p-m-1}
\frac{(2m+2k+1)
}
{k!(p-m-k)(p+m+k+1)}
\left[1+\frac{k!(p+m)!}{(2m+k)!(p-m)!}\right]
Q_{m-\frac12}^{m+k+\frac12}(\chi)
\nonumber\\
&&\hspace{-1.0cm} +\frac{1}{(p-m)!}
\sum_{k=0}^{m-1}
\frac{2k+1}{(m+k)!(p-k)(p+k+1)}Q_{m-\frac12}^{k+\frac12}(\chi)\nonumber\\
&&\hspace{-1.0cm}=(-1)^{m+\gdt-1}2^{\gdt-1}\left(\frac{rr'}{RR'}\right)^p
(\chi^2-1)^{-\frac{p}2-\frac14}
\left(\frac{r_>^2-r_<^2}{2rr'}\right)^{p+\frac{d-1}2}
\nonumber\\
&&\hspace{-0.5cm}\times
\Biggl\{\ 
\sum_{m_2=0}^{p-m}
\epsilon_{m_2}\cos(m_2(\phi_2-\phi_2'))\cdots
\sum_{m_j=0}^{p-\sum_{i=1}^{j-1}m_i}\epsilon_{m_j}\cos(m_j(\phi_j-\phi_j'))\cdots
\sum_{m_\gdt=0}^{p-\sum_{i=1}^{\gdt-1}m_i}\epsilon_{m_\gdt}\cos(m_\gdt(\phi_\gdt-\phi_\gdt'))
\nonumber\\
&&\hspace{0.0cm}\times
(-1)^M\sum_{n_{\gdt-1}=0}^{\lfloor\frac{p-M}{2}\rfloor}
\PsiH{\gdt-1}{\log_2d}{m_2,m}\cdots
\sum_{n=0}^{\lfloor\frac{p+2n-2N-M}{2}\rfloor}
\PsIH{1}{\log_2d}{l_2,l_3}
\nonumber\\
&&\hspace{0.0cm}\times\biggr(\ 
\frac{2\log r_>+2H_{2p+d-2}+H_p-H_{p+\gdt-1}-H_{p+M+2N+d-2}-H_{p-2N-M}}
{(p-2N-M)!(p+2N+M+d-2)!}
Q_{2N+M+\frac{d-3}2}^{p+\frac{d-1}2}(\zeta)
\nonumber\\
&&\hspace{0.5cm}+\frac{1}{(p+2N+M+d-2)!}
\sum_{k=0}^{p-2N-M-1}
\frac{(4N+2M+2k+d-1)}
{k!(p-k-2N-M)(p+k+2N+M+d-1)}\nonumber\\
&&\hspace{1.0cm}\times 
\left[1+\frac{k!(p+2N+M+d-2)!}{(k+4N+2M+d-2)!(p-2N-M)!}\right]
Q_{2N+M+\frac{d-3}2}^{k+2N+M+\frac{d-1}2}(\zeta)
\nonumber\\
&&\hspace{0.5cm}+\frac{1}{(p-2N-M)!}
\sum_{k=0}^{2N+M+\gdt-2}
\frac{2k+1}{(2N+M+k+\gdt-1)!(p+\gdt-k-1)(p+k+\gdt)}
Q_{2N+M+\frac{d-3}2}^{k+\frac12}(\zeta)\ 
\biggr)
\nonumber\\
&&\hspace{0.0cm}+(-1)^{p+1}
\sum_{m_2=0}^\infty\epsilon_{m_2}\cos(m_2(\phi_2-\phi_2'))\cdots
\sum_{m_\gdt=0}^\infty\epsilon_{m_\gdt}\cos(m_\gdt(\phi_\gdt-\phi_\gdt'))
\nonumber\\
&&\hspace{0.5cm}
\times\sum_{n_{\gdt-1}=0}^\infty
\PsiH{\gdt-1}{\log_2d}{m_2,m}\cdots
\sum_{n_2=0}^\infty
\PsiH{2}{\log_2d}{l_4,l_5}
\nonumber\\
&&\hspace{0.5cm}\times\sum_{n=\max\left(0,\lfloor\frac{p-2N-M}{2}\rfloor+1\right)}^\infty
\frac{(2N+M-p-1)!}
{(2N+M+p+d-2)!}
\PsIH{1}{\log_2d}{l_2,l_3}
Q_{2N+M+\frac{d-3}2}^{p+\frac{d-1}2}(\zeta)
\Biggr\}.
\label{thm13it}
\end{eqnarray}
\end{thm}
\begin{proof}
{
Again, we start by comparing the Fourier and Gegenbauer expansions of the logarithmic fundamental solutions.  For the Gegenbauer expansion, replace $\mu$ with $\frac{d}{2}-1$, use \eqref{zetasquare} and the following relation:
\begin{equation*}
    \log(r r') + \log \left( \zeta + \sqrt{ \zeta^2-1} \right) = 2 \log r_>.
\end{equation*}
The Gegenbauer expansion of Theorem \ref{gegexpand} can be rewritten in Hopf coordinates by use of \eqref{batgegharm}.
Next, we consider the multi-sums. The first sums of $\sum_{l=0}^p \sum_K$, are dealt with by Lemma \ref{sumreversal}. The later sum uses Lemma \ref{sumHopfinfty}.  Using the identity $l = 2N+M$ and compare the Fourier coefficients to complete the proof.
}
\end{proof}


\noindent The simplest example of a {logarithmic addition theorem for $0 \le m\le p$ in generalized Hopf coordinates occurs in four dimensions}.

\begin{cor}
Let $0\le m\le p$, $m\in\N_0$, $p=k-d/2\in\N_0$ where $d=4$. Then
\begin{eqnarray*}
&&\hspace{-1.5cm} \frac{1}{(p+m)!(p-m)!}
\left(\log(RR')
+\log(\chi+\sqrt{\chi^2-1})+2H_{2p}-H_{p+m}-H_{p-m}
\right)
Q_{m-\frac12}^{p+\frac12}(\chi)\\
&&\hspace{-1.25cm} +\frac{1}{(p+m)!}
\sum_{k=0}^{p-m-1}
\frac{(2m+2k+1)
}
{k!(p-m-k)(p+m+k+1)}
\left[1+\frac{k!(p+m)!}{(2m+k)!(p-m)!}\right]
Q_{m-\frac12}^{m+k+\frac12}(\chi)
\\
&&\hspace{-1.25cm} +\frac{1}{(p-m)!}
\sum_{k=0}^{m-1}
\frac{2k+1}{(m+k)!(p-k)(p+k+1)}Q_{m-\frac12}^{k+\frac12}(\chi)\\
&&\hspace{-1.0cm}=-2\left(\frac{rr'}{RR'}\right)^p(\chi^2-1)^{-\frac{p}2-\frac14}
\left(\frac{r_>^2-r_<^2}{2rr'}\right)^{p+\frac32}(\sin\vartheta\sin\vartheta')^m
\\
&&\hspace{-0.5cm}\times
\Biggl\{\ \sum_{m_2=0}^{p-m}
\epsilon_{m_2}\cos(m_2(\phi_2-\phi_2'))(\cos\vartheta\cos\vartheta')^{m_2}(-1)^{m_2}
\\
&&\hspace{0.0cm}\times\sum_{n=0}^{\lfloor\frac{p-m-m_2}{2}\rfloor}
\frac{(2n+m+m_2+1)(n+m+m_2)!n!}{(n+m)!(n+m_2)!}
P_n^{(m,m_2)}(\cos(2\vartheta))
P_n^{(m,m_2)}(\cos(2\vartheta'))
\\
&&\hspace{0.0cm}\times\biggr(\ 
\frac{2\log r_>+2H_{2p+2}+H_p-H_{p+1}-H_{p+2n+m+m_2+2}-H_{p-2n-m-m_2}}
{(p-2n-m-m_2)!(p+2n+m+m_2+2)!}
Q_{2n+m+m_2+\frac12}^{p+\frac32}(\zeta)
\\
&&\hspace{0.5cm}+\frac{1}{(p+2n+m+m_2+2)!}
\sum_{k=0}^{p-2n-m-m_2-1}
\frac{(4n+2m+2m_2+2k+3)
}
{k!(p-k-2n-m-m_2)(p+k+2n+m+m_2+3)}\\
&&\hspace{1.0cm}\times 
\left[1+\frac{k!(p+2n+m+m_2+2)!}{(k+4n+2m+2m_2+2)!(p-2n-m-m_2)!}\right]
Q_{2n+m+m_2+\frac12}^{k+2n+m+m_2+\frac32}(\zeta)
\\
&&\hspace{0.5cm}+\frac{1}{(p-2n-m-m_2)!}
\sum_{k=0}^{2n+m+m_2}
\frac{2k+1}{(2n+m+m_2+k+1)!(p-k+1)(p+k+2)}
Q_{2n+m+m_2+\frac12}^{k+\frac12}(\zeta)\ 
\biggr)
\\
&&\hspace{0.0cm}+(-1)^{p+m+1}\sum_{m_2=0}^\infty\epsilon_{m_2}\cos(m_2(\phi_2-\phi_2'))
(\cos\vartheta\cos\vartheta')^{m_2}
\\
&&\hspace{0.5cm}\times\sum_{n=\max\left(0,\lfloor\frac{p-m-m_2}{2}\rfloor+1\right)}^\infty
\frac{(2n+m+m_2+1)(n+m+m_2)!(2n+m+m_2-p-1)!n!}
{(n+m)!(n+m_2)!(2n+m+m_2+p+2)!}\\
&&\hspace{4.0cm}\times P_n^{(m,m_2)}(\cos(2\vartheta))
P_n^{(m,m_2)}(\cos(2\vartheta'))
Q_{2n+m+m_2+\frac12}^{p+\frac32}(\zeta)
\Biggr\}.
\end{eqnarray*}
\end{cor}
\begin{proof}
{
When $d=4$, note that $N=n$ and $M = m + m_2$.  The functions of the Hopf coordinate system
can be simplified as in Theorem \ref{biHopf}.
}
\end{proof}


\noindent {The general logarithmic addition theorem 
for $m\ge p+1$ in generalized Hopf 
coordinates is given as follows.}

\begin{thm}
Let $p,m\in\N_0$, $m\ge p+1$, $d=2^q$, $q\ge 2$.
Then
\begin{eqnarray}
&&\hspace{-1.25cm}Q_{m-\frac12}^{p+\frac12}(\chi)=(-2)^{\gdt-1}2^{\gdt-1}\frac{(p+m)!}{(m-p-1)!}
(\chi^2-1)^{-\frac{p}{2}-\frac14}
\rrRRp \pdh\nonumber\\[0.2cm]
&&\hspace{-0.5cm}\times \sum_{m_2=0}^{\infty}\epsilon_{m_2}\cos(m_2(\phi_2-\phi_2'))\times
\cdots\times
\sum_{m_\gdt=0}^{\infty} \epsilon_{m_\gdt}\cos(m_\gdt(\phi_\gdt-\phi_\gdt'))
\sum_{n_{\gdt-1}=0}^{\infty} \PsiH{\gdt-1}{\log_2d}{m_2,m}\nonumber\\[0.2cm]
&&\hspace{-0.00cm}\times\cdots\times\sum_{n=0}^{\infty} \PsIH{1}{\log_2d}{l_2,l_3}
\frac{(2N+M-p-1)!}{(p+2N+M+d-2)!}
Q_{2N+M+\frac{d-3}{2}}^{p+\frac{d-1}{2}}(\zeta).
\end{eqnarray}
\end{thm}
\begin{proof}
{Similar to Theorem \ref{logHopf}.  For this case only the Fourier coefficients where $m \ge p+1$ are used, therefore Lemma \ref{multisumZ3} is needed.  Simplifying the expression completes the proof.}
\end{proof}

\noindent The simplest example of a {logarithmic addition theorem for $m\ge p+1$ in generalized Hopf coordinates occurs in four dimensions}.

\begin{thm}
Let $p,m\in\N_0$, $m\ge p+1$.
Then
\begin{eqnarray}
&&\hspace{-1cm}Q_{m-\frac12}^{p+\frac12}(\chi)=-2\frac{(p+m)!}{(m-p-1)!}
(\chi^2-1)^{-\frac{p}{2}-\frac14}
\rrRRp \pth(\sin\vartheta\sin\vartheta')^m\nonumber\\[0.2cm]
&&\hspace{-0.5cm}\times \sum_{m_2=0}^{\infty}\epsilon_{m_2}\cos(m_2(\phi_2-\phi_2'))
(\cos\vartheta\cos\vartheta')^{m_2}
\sum_{n=0}^{\infty}
P_n^{(m,m_2)}(\cos(2\vartheta))
P_n^{(m,m_2)}(\cos(2\vartheta'))
\nonumber\\[0.2cm]
&&\hspace{-0.5cm}\times
\frac{(2n+m+m_2+1)(2n+m+m_2-p-1)!(n+m+m_2)!n! }
{(p+2n+m+m_{{2}}+2)!(n+m)!(n+m_2)!}
Q_{2n+m+m_2+\frac12}^{p+\frac32}(\zeta).
\end{eqnarray}
\end{thm}
\begin{proof}
{Simplifying as in Theorem \ref{logHopf}  for $d=4$ completes 
the proof.}
\end{proof}

%
%
%

\subsection*{Acknowledgements}
We would like to thank A.F.M.~Tom ter Elst, Matthew Auger
and Rados{\l}aw Szmytkowski for valuable discussions.  


\def\cprime{$'$} \def\dbar{\leavevmode\hbox to 0pt{\hskip.2ex \accent"16\hss}d}

\end{document}